\documentclass[10pt]{amsart}
\usepackage{amsmath}
\usepackage{amsthm}
\usepackage{amscd}
\usepackage{amssymb}
\usepackage{amsfonts}
\usepackage{graphics}
\usepackage{color}
\usepackage{enumerate}
\usepackage[english]{babel}
\usepackage[utf8]{inputenc}

\usepackage{graphicx}
\usepackage{caption}
\usepackage{subcaption}
\usepackage{hyperref}
\date{}

\makeatletter
\def\@tocline#1#2#3#4#5#6#7{\relax
  \ifnum #1>\c@tocdepth 
  \else
    \par \addpenalty\@secpenalty\addvspace{#2}%
    \begingroup \hyphenpenalty\@M
    \@ifempty{#4}{%
      \@tempdima\csname r@tocindent\number#1\endcsname\relax
    }{%
      \@tempdima#4\relax
    }%
    \parindent\z@ \leftskip#3\relax \advance\leftskip\@tempdima\relax
    \rightskip\@pnumwidth plus4em \parfillskip-\@pnumwidth
    #5\leavevmode\hskip-\@tempdima
      \ifcase #1
       \or\or \hskip 1em \or \hskip 2em \else \hskip 3em \fi%
      #6\nobreak\relax
    \dotfill\hbox to\@pnumwidth{\@tocpagenum{#7}}\par
    \nobreak
    \endgroup
  \fi}
\makeatother

\theoremstyle{plain}
\newtheorem{theorem}{Theorem}[section]
\newtheorem{proposition}[theorem]{Proposition}
\newtheorem{lemma}[theorem]{Lemma}

\theoremstyle{definition}
\newtheorem{definition}[theorem]{Definition}

\newtheorem{remark}[theorem]{Remark}

\theoremstyle{remark}

\newcommand{\na}{\mathbb{N}}
\newcommand{\re}{\mathbb{R}}

\def\huz{H^1_0 (\Omega)}

\def\lio{{L^{\infty}(\Omega)}}

\def\into{\int_{\Omega}}

\def\dys{\displaystyle}

\numberwithin{equation}{section}

\title[Continuum of bifurcation points with no emanating continua]{Intervals of Bifurcation points for semilinear Elliptic problems}

\author[J. Carmona Tapia]{Jos\'{e} Carmona Tapia}
\address{Departamento de Matem\'aticas\\ Uni\-ver\-si\-dad de Alme\-r\'ia\\Ctra. Sacramento s/n\\
La Ca\~{n}ada de San Urbano\\ 04120 - Al\-me\-r\'{\i}a, Spain}
\email{jcarmona@ual.es}

\author[A. J. Mart\'{i}nez Aparicio]{Antonio J. Mart\'{i}nez Aparicio}
\address{Departamento de Matem\'aticas\\ Uni\-ver\-si\-dad de Alme\-r\'ia\\Ctra. Sacramento s/n\\
La Ca\~{n}ada de San Urbano\\ 04120 - Al\-me\-r\'{\i}a, Spain}
\email{ajmaparicio@ual.es}

\author[P. J. Mart\'{i}nez-Aparicio]{Pedro J. Mart\'{i}nez-Aparicio}
\address{Departamento de Matem\'aticas\\ Uni\-ver\-si\-dad de Alme\-r\'ia\\Ctra. Sacramento s/n\\
La Ca\~{n}ada de San Urbano\\ 04120 - Al\-me\-r\'{\i}a, Spain}
\email{pedroj.ma@ual.es}


\begin{document}

\maketitle

\begin{abstract}
In this paper, we study the behavior of multiple continua of solutions to the semilinear elliptic problem
\begin{equation*}
    \begin{cases}
   -\Delta u = \lambda f(u) &\text{ in } \Omega,\\
   u=0 &\text{ on } \partial\Omega,
    \end{cases}
\end{equation*}
where $\Omega$ is a bounded open subset of $\re^N$ and $f$ is a nonnegative continuous real function with multiple zeros. We analyze both the behavior of unbounded continua of solutions having norm between consecutive zeros of $f$, and the asymptotic behavior of the multiple unbounded continua in the case in which $f$ has a countable infinite set of positive zeros. In both cases, we pay special attention to the multiplicity results they give rise to. For the model cases $f(t) = t^r(1+\sin t)$ and $f(t) = t^r \left(1+\sin \frac{1}{t}\right)$ with $r\geq 0$ we show the surprising fact that there are some values of $r$ for which every $\lambda>0$ is a bifurcation point (either from infinity or from zero) that is not a branching point.
\end{abstract}

\tableofcontents

\section{Introduction and statement of the main results}

We consider the following semilinear elliptic boundary value problem
\begin{equation}
    \tag{$P_\lambda$}
    \label{eq:Problema_general}
    \begin{cases}
   -\Delta u = \lambda f(u) &\text{ in } \Omega,\\
   u=0 &\text{ on } \partial\Omega,
    \end{cases}
\end{equation}
where $\Omega$ is a smooth open and bounded subset of $\re^N$ and $f$ is a nonnegative continuous real function having a countable set of positive zeros, either finite or infinte. We are interested in positive solutions of~\eqref{eq:Problema_general} and thus we may assume that $f(s) = f(0) \geq 0$ for every $s \leq 0$ which implies that any solution is nonnegative. 

A very useful technique used in the literature  to describe the solution set of bounded positive solutions of \eqref{eq:Problema_general} is to study the existence and global behavior of connected and closed subsets (continua) in 
\[
\mathcal{S}=\{(\lambda,u)\in \mathbb{R}_0^+\times L^\infty(\Omega): u \textrm{ solves } \eqref{eq:Problema_general}\}.
\]
The trivial set $\mathcal{T}=[0,+\infty) \times \{ 0 \}$ is such a continuum when $f(0)=0$. In this case, as a result of the pioneering work \cite{Rab} in this type of method, existence of continua of nontrivial solutions is related to bifurcation points from the line of trivial solutions, that is, points in $\mathcal{B}_0=\mathcal{T}\cap \overline{\mathcal{S}\setminus \mathcal{T}}$. Even more, it is also related to bifurcation points from infinity, that is, points in $\mathcal{B}_\infty=\mathcal{T}\cap \overline{\mathcal{S}_\infty}$, where
\[
\mathcal{S}_\infty = \left\{(\lambda,u)\in \mathbb{R}_0^+\times L^\infty(\Omega): u\ne 0, \frac{u}{\|u\|^2_{L^\infty(\Omega)}} \textrm{ solves } \eqref{eq:Problema_general} \right\}.
\]
The set $\mathcal{B}_0$ has been well studied when $f(0)=0$ and $f$ is strictly positive in $(0,\varepsilon)$ for some $\varepsilon>0$ and depends on the behavior of $f$ at zero. Thus, when 
$f$ is sublinear, superlinear or asymptotically linear at zero (i.e. $\lim_{s\to 0^+} {f(s)}/{s} = 0$, $\lim_{s\to 0^+} {f(s)}/{s} = +\infty$ or $\lim_{s\to 0^+} {f(s)}/{s} := m_0>0$) then, respectively, $\mathcal{B}_0=\emptyset$, $\mathcal{B}_0=\{(0,0)\}$ or $\mathcal{B}_0=\{(\lambda_0,0)\}$ for some $\lambda_0>0$. Furthermore, in the case $\mathcal{B}_0\ne \emptyset$, using \cite{Rab}, the connected component of $\overline{\mathcal{S}\setminus \mathcal{T}}$ containing $\mathcal{B}_0$ is an unbounded continuum (global bifurcation) $\Sigma_0$.  In \cite{Rab} it is also proved, whenever $f(0)>0$, that $\mathcal{B}_0=\{(0,0)\}$ and there exists an unbounded continuum $\Sigma\subset \mathcal{S}$ with $(0,0)\in \Sigma$.

Analogously, whenever $f$ is positive in $(M,+\infty)$ for some $M>0$ and it is sublinear, superlinear (and subcritical) or asymptotically linear at infinity (i.e. $\lim_{s\to +\infty} {f(s)}/{s} = 0$, $\lim_{s\to +\infty} {f(s)}/{s} = +\infty$ or $\lim_{s\to +\infty} {f(s)}/{s} := m_\infty>0$) then, respectively, $\mathcal{B}_\infty=\emptyset$, $\mathcal{B}_\infty=\{(0,0)\}$ or $\mathcal{B}_\infty=\{(\lambda_\infty,0)\}$ for some $\lambda_\infty>0$.  Also, in the case $\mathcal{B}_\infty\ne \emptyset$, global bifurcation may occur. 

We recommend the reader the interesting survey \cite{Lions} with a great content where different behaviors at zero and infinity are considered. 

We remark explicitly that in the previous cases $\mathcal{B}_0$ and $\mathcal{B}_\infty$ have at most one element. When $f$ is asymptotically linear either at zero or at infinity, since we are looking for positive solutions, $\lambda_0$ and $\lambda_\infty$ correspond, see \cite{Amb-Hess}, with the first eigenvalue for a weighted eigenvalue problem.

When global bifurcation occurs and $f$ has an oscillatory asymptotically linear behavior at infinity, then the global behavior of the continuum may lead to an interval of bifurcation points at infinity as in \cite{Rynne} and \cite{XSOW}. In this case $\mathcal{B}_\infty$ contains a bounded segment in $\mathcal{T}$.

As a consequence of the results proved in \cite{correa1} it may also be deduced that $\mathcal{B}_\infty$ contains a segment in $\mathcal{T}$ for particular functions $f$ having a divergent sequence of zeros.  This phenomenon of a whole interval of bifurcation points was first discovered in \cite{BenciFortunato} for problems in $\mathbb{R}^N$, see also \cite{Stuart} for non uniformly elliptic operators in bounded domains.

As far as we know, general existence results of an unbounded interval of bifurcation points, either at zero or at infinity,  for positive solutions to uniformly elliptic problems in bounded domains are unknown.


We now analyze the case where $f$ has a positive zero. Let us recall that if for some $0<\sigma$, $f(\sigma)=0$ and there exists $M > 0$ such that
\begin{equation}
    \label{eq:hip_f_creciente}
    \tag*{(\theequation)$_\sigma$}
    \refstepcounter{equation}
    f(s)+M s \text{ is increasing for } s\in(0,\sigma),
\end{equation}
then, as a consequence of the strong maximum principle, it is well known that 
\begin{equation}\label{equ:noexis}
 \mathcal{S} \cap \{(\lambda,u)\in \mathbb{R}\times L^\infty(\Omega): \|u\|_{L^\infty(\Omega)}=\sigma\}=\emptyset.
\end{equation}
Therefore, for locally Lipschitz functions $f$ with zeros, when $(\lambda,u)$ is located in a connected component of $\mathcal{S}\setminus  \mathcal{T}$ then $\|u\|_{L^\infty(\Omega)}$ lies in an open interval where $f$ is strictly positive. In the case in which this maximal interval is $(\alpha, \beta)$, for some   $0<\alpha<\beta<+\infty$ positive real zeros of $f$,  those connected component are continua in 
\[
\mathcal{S}_{\alpha\beta}=\{(\lambda,u)\in \mathbb{R}_0^+\times L^\infty(\Omega): u \textrm{ solves } \eqref{eq:Problema_general} \textrm{ and } \alpha<\|u\|_{L^\infty(\Omega)}<\beta\}.
\]
Using again the arguments in \cite{Rab}, for locally Lipschitz functions $f$, vertically isolated solutions in $\mathcal{S}_{\alpha\beta}$ with nontrivial Leray-Schauder index give rise to two different unbounded continua in $\mathcal{S}_{\alpha\beta}$. This is the case for $N=1$, see \cite{correa2}, where the uniqueness of solutions with given $L^\infty(\Omega)$ norm can be easily deduced. For $N>1$, existence and multiplicity of solutions for large $\lambda$ are well known since the pioneering work \cite{Hess}, see also \cite{Brown-Budin} for the ordinary case. However, although this multiplicity result sugests the existence of an unbounded continuum with two branches of different solutions ($\subset$-shape), see \cite{Lions, Liu}, existence and global behavior of continua in $\mathcal{S}_{\alpha\beta}$ has been less studied when $N>1$.

Existence and multiplicity results have been extended recently to nonlocal operators in \cite{A-C-MA} and \cite{CR}. Also the behavior of the solutions of~\eqref{eq:Problema_general} when $\lambda$ tends to infinity has been analyzed in~\cite{B-GM-I} and~\cite{GM-I}.



The purpose of this paper is to get a global picture of the solution set $\mathcal{S}$ of~\eqref{eq:Problema_general} for a general nonnegative continuous function $f$ with zeros. Our main objective is to clearly describe how the set of zeros of $f$, in all possible cases, affects to the solution set $\mathcal{S}$ and the sets of bifurcation points $\mathcal{B}_0$ and $\mathcal{B}_\infty$ and also describe, when possible, the continua of solutions in $\mathcal{S}$, particularly those in $\mathcal{S}_{\alpha\beta}$ and providing results of great interest with our study.


%

First, we show that condition~\ref{eq:hip_f_creciente}, a condition implicitly included in all the above mentioned works, can be weakened to obtain \eqref{equ:noexis}. Namely, we assume only that there exist $M > 0$ and $\varepsilon>0$ depending on $\sigma$ such that 
\begin{equation}
    \label{eq:hip_f_creciente_3}
    \refstepcounter{equation}
    f(s)+M s \text{ is increasing for } s\in(\sigma-\varepsilon,\sigma).
    \tag*{(\theequation)$_\sigma$}
\end{equation}
An example of function satisfying \ref{eq:hip_f_creciente_3} and not \ref{eq:hip_f_creciente} is $f(s)=|1-\sqrt s|$ with $\sigma=1$. Subsequently, existence of solution in $\mathcal{S}_{\alpha\beta}$ is proved, for large $\lambda$, when this condition \ref{eq:hip_f_creciente_3} is satisfied for $\sigma=\alpha$ and $\sigma=\beta$.


We also prove the existence of a continuum of solutions in $\mathcal{S}_{\alpha\beta}$ which is unbounded on the $\lambda$-axis having a $\subset$-shape when \ref{eq:hip_f_creciente} is satisfied for $\sigma=\beta$. Moreover, the projection of this continuum on the $\lambda$-axis gives the maximal interval $[\lambda_{\alpha\beta}, +\infty)$ for the existence of solution with norm in $(\alpha, \beta)$. To our knowledge, this result was unknown for $N>1$. Moreover, we complete the global picture of the solution set $\mathcal{S}$ by describing the existence of continua of solutions with norm below the first positive zero or above the last one, in the case that these are positive real numbers. 


Secondly, and no less important, we consider the case in which $f$ has a monotone sequence $\{\alpha_n\}$ of positive zeros and we study the asymptotic behavior of the infinitely many solution sets $\mathcal{S}_n$ between each pair of zeros ($\mathcal{S}_n\equiv \mathcal{S}_{\alpha_n\alpha_{n+1}}$ or $\mathcal{S}_n\equiv \mathcal{S}_{\alpha_{n+1}\alpha_n}$). We show that, in the case $\alpha_n\to \alpha\in (0,+\infty)$, independently of $f$, the sequence 
\begin{equation}\label{equ:lambda:n}
\lambda_n=\inf\{\lambda\in (0,\infty): (\lambda,u)\in \mathcal{S}_n \textrm{ for some $u$}\}
\end{equation}
diverges, while in the case $\alpha=0$ or $\alpha=+\infty$, depending on $f$, the sequence $\lambda_n$ may diverge or not. In this later case, when $\lambda_n$ does not diverge, then an unbounded interval of bifurcations points arises. We illustrate this more in depth for the  problem~\eqref{eq:Problema_general} with $f(t) = t^r(1+\sin t)$ and $f(t) = t^r \left(1+\sin \frac{1}{t}\right)$ where $r > 0$. We obtain, for some values of $r$ a whole interval of bifurcation points, either from zero or from infinity. 
We remark that, in order to prove our main results, we deal with several properties of the solutions that, in our opinion, are very interesting in the literature of this type of problems.

All the results that we present in this work are expected to have a great impact on the literature on these type of problems, and we believe that they will be very useful in the research of related problems. 

The first result of this paper, related to the existence of a continuum with $\subset$-shape, is stated as follows. 
%
\begin{theorem}
\label{teor:continuo}
Let $f$ be a nonnegative function and let $0<\alpha<\beta$ be two consecutive zeros of $f$. Assume
that~\ref{eq:hip_f_creciente_3} is satisfied for $\sigma=\alpha$ and $\sigma=\beta$. Let us also denote 
$\lambda_{\alpha\beta}=\inf\{\lambda\in (0,\infty): (\lambda,u)\in \mathcal{S}_{\alpha\beta} \textrm{ for some $u$}\}$. Then $\lambda_{\alpha\beta}>0$ and for every $\lambda\geq \lambda_{\alpha\beta}$ there exists $(\lambda,u)\in \mathcal{S}_{\alpha\beta}$. Moreover, if~\ref{eq:hip_f_creciente} holds true for $\sigma=\beta$, then there exists an unbounded continuum $\Sigma \subset \mathcal{S}_{\alpha\beta}$ with $\subset$-shape such that:
\begin{enumerate}[i)]
\item $\lambda_{\alpha\beta} = \min\left( \mathrm{Proj}_{\lambda}\Sigma \right)$.
\item For every $\lambda> \lambda_{\alpha\beta}$ there exist at least two solutions $u_1, u_2$ of \eqref{eq:Problema_general}  such that $u_1\not\equiv u_2$ and
$(\lambda,u_1),(\lambda, u_2)\in \Sigma$. 
\end{enumerate}
\end{theorem}

We observe that this theorem can be applied to every pair of consecutive positive zeros of $f$.
 In the case where the function $f$ has a monotone sequence of positive zeros $\alpha_n \to \alpha$, a natural question that arises is: what is the behavior of the sequence $\{\lambda_n\}$ given by \eqref{equ:lambda:n}?
 
 We show that when $0<\alpha<\infty$, then $\{\lambda_n\}$ always diverges (see Proposition~\ref{prop:acum_pos}). When $\alpha=0$ or $\alpha=+\infty$, then $\{\lambda_n\}$ can behave in several ways, some of them implying the existence of infinitely many bifurcation points (either from zero if $\alpha=0$ or from infinity if $\alpha=+\infty$) that are not branching points.

To illustrate these latter situations, we study in depth two wide families of different problems in which the positive zeros behave very differently. More precisely, we consider the sequences of zeros of the functions $(1+\sin u)$ or $\left(1+\sin \frac1 u\right)$ on the basis of the positive functions $u^rg(u)$ with 
\begin{equation}
\label{hipg}
0< g\in \mathcal{C}^1(\mathbb{R}) \mbox{ a bounded function such that } \inf_{s\geq 0} g(s) >0.
\end{equation}
Thus, due to the hypothesis \eqref{hipg} on $g$, the superlinear or sublinear profile of $f$ is determined by the profile of $u^r$ and the function $g$ may have an oscillatory behavior at infinity or at zero which  plays an important role in the case $r=1$.

The first family of problems is
\begin{equation}
    \tag{$Q_\lambda$}
    \label{eq:Problema_seno_u}
    \begin{cases}
   -\Delta u = \lambda u^r(1+\sin u)g(u) &\text{ in } \Omega,\\
   u=0 &\text{ on } \partial\Omega,
    \end{cases}
\end{equation}
with $r\geq 0$.

 Here, the sequence of positive zeros of $f$ diverges. In this case, we have the following existence and multiplicity result depending on $r$, which is a consequence of the behavior of the sequence $\{\lambda_n\}$ mentioned above.

\begin{theorem}
\label{teor:seno_u}
Assume that \eqref{hipg} holds true and let $n(\lambda)$ denote the number of solutions of~\eqref{eq:Problema_seno_u} for $\lambda>0$ fixed, then:
    \begin{enumerate}[i)]
        \item If $0\leq r<1$, then for every $m\in\na$ there exists some $\tilde{\lambda}_m>0$ such that $n(\lambda)\geq m$ for every $\lambda>\tilde{\lambda}_m$. In addition, no $\lambda>0$ is a bifurcation point from infinity.
        \item If $r=1$, then there exists some $\tilde{\lambda}_\infty>0$ such that $n(\lambda) = \infty$ for every $\lambda>\tilde{\lambda}_\infty$. Moreover, every $\lambda\geq \tilde{\lambda}_\infty$ is a bifurcation point from infinity but not a branching point from infinity. In addition, every $\lambda<\tilde \lambda_\infty$ is not a bifurcation point from infinity.
        \item If $r>1$, then $n(\lambda)=\infty$ for every $\lambda>0$. Furthermore, every $\lambda>0$ is a bifurcation point from infinity but not a branching point from infinity.
    \end{enumerate}
\end{theorem}

\begin{figure}[ht]
\label{fig:seno_u}
\centering
\begin{subfigure}{.33\textwidth}
  \centering
  \includegraphics[scale=0.43]{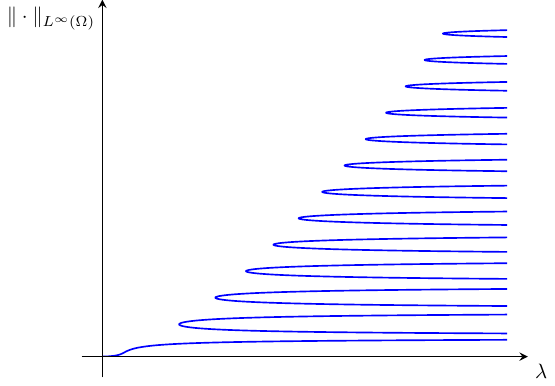}
  \caption{$r<1$}
\end{subfigure}%
\begin{subfigure}{.33\textwidth}
  \centering
  \includegraphics[scale=0.43]{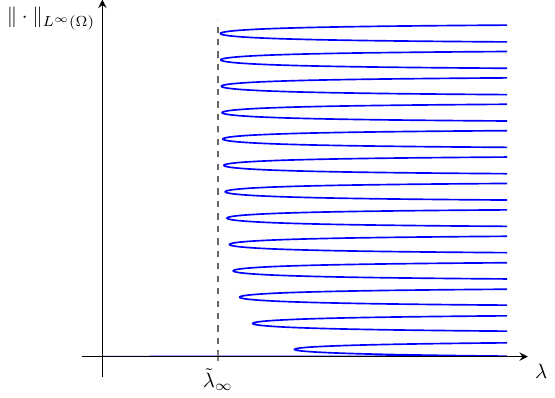}
  \caption{$r=1$}
\end{subfigure}%
\begin{subfigure}{.33\textwidth}
  \centering
  \includegraphics[scale=0.43]{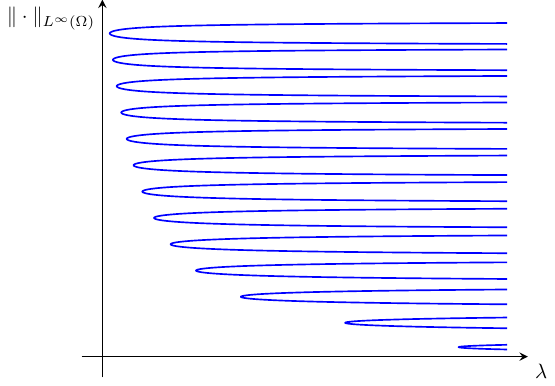}
  \caption{$r>1$}
\end{subfigure}
\caption{Sketch of the continua diagrams of problem~\eqref{eq:Problema_seno_u} depending on the value of $r$ on the basis of Theorem~\ref{teor:continuo} and Theorem~\ref{teor:seno_u}.
}
\end{figure}
We remark explicitly that, for $r\geq 1$, the function $f(s)=s^r(1+\sin s)$ does not satisfies the conditions of \cite{correa1}.

The second family of problems is
\begin{equation}
    \tag{$R_\lambda$}
    \label{eq:Problema_seno_1/u}
    \begin{cases}
 \dys   -\Delta u = \lambda u^r\left(1+\sin \frac{1}{u}\right) g(u) &\text{ in } \Omega,\\
    u=0 &\text{ on } \partial\Omega,
    \end{cases}
\end{equation}
with $r>0$ and satisfying \eqref{hipg}. Observe that in these problems the sequence of positive zeros of $f$ converges to 0 and, for $r<2$, $f$ does not satisfies \ref{eq:hip_f_creciente} at any of them. We have the following existence and multiplicity result.

\begin{theorem}
\label{teor:seno_1/u}
   Assume that \eqref{hipg} holds true and let  $n(\lambda)$ denote the number of solutions of~\eqref{eq:Problema_seno_1/u} for $\lambda>0$ fixed, then:
    \begin{enumerate}[i)]
        \item If $0< r<1$, then $n(\lambda)=\infty$ for every $\lambda>0$. Furthermore, every $\lambda>0$ is a bifurcation point from zero but not a branching point from zero.
        \item If $r=1$, then there exists some $\tilde{\lambda}_0>0$ such that $n(\lambda) = \infty$ for every $\lambda>\tilde{\lambda}_0$. Moreover, every $\lambda\geq \tilde{\lambda}_0$ is a bifurcation point from zero but not a branching point from zero. In addition, every $\lambda<\tilde \lambda_0$ is not a bifurcation point from zero.
        \item If $r>1$, then for every $m\in\na$ there exists some $\tilde{\lambda}_m>0$ such that $n(\lambda)\geq m$ for every $\lambda>\tilde{\lambda}_m$. In addition, no $\lambda>0$ is a bifurcation point from zero.
    \end{enumerate}
\end{theorem}

\begin{figure}[ht]
\label{fig:seno_1/u}
\centering
\begin{subfigure}{.33\textwidth}
  \centering
  \includegraphics[scale=0.43]{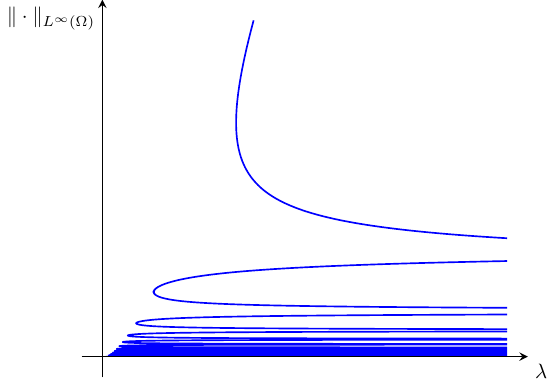}
  \caption{$r<1$}
\end{subfigure}%
\begin{subfigure}{.33\textwidth}
  \centering
  \includegraphics[scale=0.43]{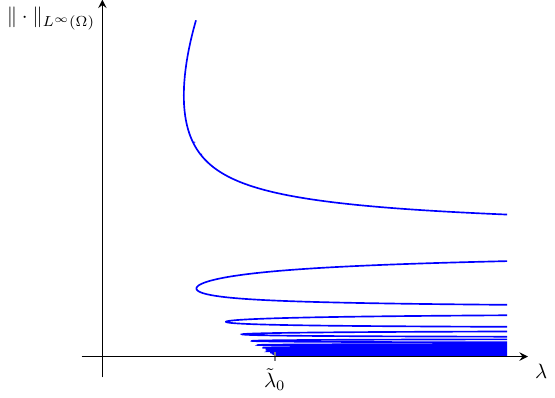}
  \caption{$r=1$}
\end{subfigure}%
\begin{subfigure}{0.33\textwidth}
  \centering
  \includegraphics[scale=0.43]{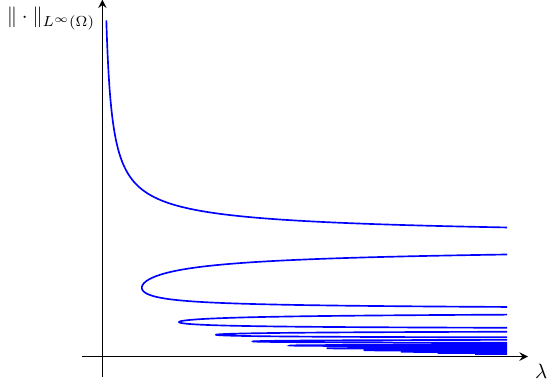}
  \caption{$r>1$}
\end{subfigure}
\caption{Sketch of the continua diagrams of problem~\eqref{eq:Problema_seno_1/u} depending on the value of $r$ on the basis of Theorem~\ref{teor:continuo} and Theorem~\ref{teor:seno_1/u}.}
\end{figure}

The plan of the paper is as follows. In Section~\ref{sec:continuo} we analyze the behavior of the solutions in $\mathcal{S}_{\alpha\beta}$. To do that, we prove several lemmas and then we prove Theorem~\ref{teor:continuo}.  Section~\ref{sec:sucesion_de_ceros} is devoted to the case in which $f$ has a monotone sequence of positive zeros. We analyze the behavior of the sequence $\{\lambda_n\}$ when $n$ goes to infinity to get a complete picture of the solution set $\mathcal{S}$ of problem \eqref{eq:Problema_general}. Moreover, we give some examples considering in Subsection~\ref{sec:seno_u} the problem~\eqref{eq:Problema_seno_u} in which the sequence of positive zeros of $f$ diverges and in Subsection~\ref{sec:seno_1/u} the problem~\eqref{eq:Problema_seno_1/u} in which the sequence of positive zeros of $f$ converges to 0. Finally, in Section~\ref{firstsecond} we detail the behavior of the solutions below the first positive zero or above the last one. We analyze which are all possible cases, which of them are solved in the literature, which are treated by us in the present work, and, to the best of our knowledge, an open problem that is still pending.

\section{Behavior of solutions between two consecutive zeros}
\label{sec:continuo}

We dedicate this section to detail the behavior of the solutions in $\mathcal{S}_{\alpha\beta}$ where $\alpha, \beta$ are two consecutive positive zeros of the nonlinearity $f$. To do that, we prove some interesting properties of the solutions of the problem \eqref{eq:Problema_general}.


Let us start with a well-known result consequence of the strong maximum principle; see, for example,~\cite{Amb-Hess}.
\begin{lemma}\label{lem:AmbHess}
 Assume that $f$ is a continuous function with $f(0)\geq 0$ and verifying~\ref{eq:hip_f_creciente} for some $\sigma>0$ such that $f(\sigma)\leq 0$. Then, for every $\lambda\geq 0$, there is no positive solution $u\in H_0^1(\Omega)\cap L^\infty(\Omega)$ of~\eqref{eq:Problema_general} such that $\|u\|_{L^\infty(\Omega)} = \sigma$.
\qed
\end{lemma}

\begin{remark}
In the next lemma, we highlight that the conclusion of Lemma~\ref{lem:AmbHess} is true when we replace condition~\ref{eq:hip_f_creciente} by the less restrictive condition~\ref{eq:hip_f_creciente_3}
which depends only on the behavior of $f$ in a left neighbourhood of $\sigma$. 
\end{remark}

\begin{lemma}
\label{lema:no_existencia}
Assume that $f$ is a continuous function with $f(0)\geq 0$ and verifying~\ref{eq:hip_f_creciente_3} for some $\sigma>0$ such that $f(\sigma)\leq 0$. Then, for every $\lambda\geq 0$, there is no positive solution $u\in H_0^1(\Omega)\cap L^\infty(\Omega)$ of~\eqref{eq:Problema_general} such that $\|u\|_{L^\infty(\Omega)} = \sigma$.
 \end{lemma}

\begin{remark}
The importance of Lemma~\ref{lema:no_existencia} lies in the variety of nonnegative functions satisfying~\ref{eq:hip_f_creciente_3} in a neighborhood of any of their positive zeros and not satisfying~\ref{eq:hip_f_creciente}. Some examples are given by $f(t)= t^r\left(1+ \sin\frac{1}{t}\right)$ when $0<r<2$. 
\end{remark}
 
 \begin{proof}
 By virtue of~\ref{eq:hip_f_creciente_3}, for some $M>0$ and $\varepsilon>0$ the function $f(s)+Ms$ is increasing in $[\sigma-\varepsilon,\sigma]$. Let us denote $\gamma=\displaystyle\max_{t\in[0,\sigma-\frac\varepsilon2]}\{f(t)+Mt\}\geq 0$ and define $\tilde{f}\in \mathcal{C}([0,\sigma])$ by
\begin{equation*}
\tilde{f}(s)=
\begin{cases}
\displaystyle  \gamma - Ms, & s\in[0,\sigma-\varepsilon],
 \\
\dys\gamma +\frac 2{\varepsilon}(f(\sigma-\frac\varepsilon2)+M(\sigma-\frac\varepsilon2)-\gamma)\left(s-\sigma+\varepsilon\right)-Ms, & s\in (\sigma-\varepsilon,\sigma-\frac\varepsilon2),
 \\
 f(s), & s>\sigma-\frac\varepsilon2.
\end{cases}
\end{equation*}
Thus, since $\gamma\geq f(\sigma-\frac\varepsilon2)+M(\sigma-\frac\varepsilon2)$ and $f(s)+Ms$ is increasing for $s\in (\sigma-\varepsilon,\sigma-\frac\varepsilon2)$, we deduce that $f(s)\leq \tilde f(s)$ for every $s\leq \sigma$ and $\tilde f$ satisfies
\[
\tilde f(s)+ \tilde M s \text{ is increasing for } s\in(0,\sigma),
\]
with $\tilde M = M + \frac2\varepsilon(\gamma - f(\sigma-\frac\varepsilon2)-M(\sigma-\frac\varepsilon2))$. 

%

We denote by $(\tilde{P}_\lambda)$ the problem~\eqref{eq:Problema_general} with $f$ replaced by $\tilde{f}$. Assume on the contrary that $u$ is a solution of~\eqref{eq:Problema_general} with $\|u\|_{L^\infty(\Omega)} = \sigma$. Clearly, $u$ is a subsolution of $(\tilde{P}_\lambda)$ and the constant function $\sigma$ is a supersolution of $(\tilde{P}_\lambda)$. By the sub and supersolution method (see~\cite{Clement-Sweers-2}), there exists a solution $\tilde{u}$ of $(\tilde{P}_\lambda)$ satisfying that  $u\leq \tilde{u} \leq \sigma$ in $\Omega$. In that case $\|\tilde{u}\|_{L^\infty(\Omega)} = \sigma$ which is a contradiction with Lemma~\ref{lem:AmbHess} and we conclude the proof.
 \end{proof}

In the next lemma we prove that the set of values of the parameter $\lambda$ for which \eqref{eq:Problema_general} admits a solution with $L^\infty(\Omega)$ norm between two consecutive nontrivial zeros of $f$ is an unbounded closed interval which does not contain the value $\lambda=0$.

Although the result is essentially known replacing Lemma~\ref{lem:AmbHess} by Lemma~\ref{lema:no_existencia}, we perform the proof for the convenience of the reader.
 
\begin{lemma}
\label{lema:existencia}
Assume that $f$ is a nonnegative continuous function with two consecutive zeros $0<\alpha<\beta$ verifying~\ref{eq:hip_f_creciente_3} for $\sigma=\alpha$ and $\sigma=\beta$. Then the set
\[
\Lambda_{\alpha\beta} = \{\lambda\geq 0: \textrm{\eqref{eq:Problema_general} admits solution $u$ with } \alpha < \|u\|_{L^\infty(\Omega)}<\beta \}
\]
is a nonempty closed set of the form $\Lambda_{\alpha\beta}=[\lambda_{\alpha\beta},+\infty)$ for some $\lambda_{\alpha\beta}>0$.
\end{lemma}

\begin{proof}
For the sake of clarity, we have divided the proof into four steps.

{\bf First step}. We show here that there exists some $\tilde{\lambda}$ such that problem~\eqref{eq:Problema_general} admits a subsolution $\underline{u}_\lambda$ with $\|\underline{u}_\lambda\|_{L^\infty(\Omega)} \in(\alpha, \beta)$ for every $\lambda>\tilde{\lambda}$. To this end, we consider the truncated problem
\begin{equation}
\label{eq:Dem_Lema_existencia_1}
\begin{cases}
\displaystyle -\Delta u = \lambda \tilde f(u) & \mbox{in} \; \Omega,\\
u = 0 & \mbox{on} \; \partial \Omega,\\
\end{cases} 
\end{equation}
where $\tilde f$ is given by 
\begin{equation*}
\tilde f(s)=\begin{cases}
0 & \mbox{ if } s<\alpha, \\
f(s) & \mbox{ if }  \alpha\leq s\leq \beta, \\
0 & \mbox{ if } s>\beta.\\
\end{cases} 
\end{equation*}
We observe that every nontrivial solution of~\eqref{eq:Dem_Lema_existencia_1} is a subsolution of \eqref{eq:Problema_general} with $\lio$ norm belonging to $(\alpha, \beta)$ due to Lemma~\ref{lema:no_existencia}. Moreover, 
solutions of~\eqref{eq:Dem_Lema_existencia_1} correspond to critical points of the energy functional
\begin{equation*}
    \tilde{I}_\lambda(u) = \frac{1}{2} \into |\nabla u|^2 - \lambda\into \tilde{F}(u),\ \forall u\in\huz,
\end{equation*}
where $\tilde{F}(s) = \int_0^s \tilde{f}(t)\, \mathrm{d}t$.  Since $\tilde F$ is continuous and bounded, the functional $\tilde I_\lambda$ is coercive and weakly lower semicontinuous for every $\lambda>0$. Therefore, we can deduce that $\tilde I_\lambda$ has a global minimum. Now we prove that $\tilde{I}_\lambda$ does not reach its minimum at the trivial function for $\lambda$ large.
%

In fact, observe that we can take $u_0\in \mathcal{C}^\infty_0(\Omega)$ such that $u_0\equiv \beta$ in $\Omega_0\subset\subset \Omega$ and thus $\int_\Omega \tilde F(u_0)=\int_{\Omega_0} \tilde F(\beta)+\int_{\Omega\setminus\Omega_0} \tilde F(u_0)>0$ since $\int_0^\beta \tilde f(s)ds > 0$ and $\tilde f$ is nonnegative. Therefore, denoting 
\[
\tilde \lambda=\frac{\frac 12 \into |\nabla u_0|^2}{\into \tilde F(u_0)},
\]
for every $\lambda>\tilde \lambda$ we have $\tilde I_\lambda (u_0)<0=I(0)$ and in this way $\tilde I_\lambda$ reach its minimum in a critical point which is nontrivial. As a consequence, for every $\lambda>\tilde{\lambda}$ we have proved the existence of a subsolution $\underline{u}_\lambda$ of~\eqref{eq:Problema_general} with $\|\underline u_\lambda\|_{L^\infty(\Omega)} \in (\alpha, \beta)$.

{\bf Second step}. Here we show that $\Lambda_{\alpha\beta}\ne \emptyset$. Since $\overline u_\lambda = \beta$ is a supersolution of~\eqref{eq:Problema_general} for every $\lambda>0$, we can use the sub and supersolution method (see~\cite{Clement-Sweers-2}) to obtain for every $\lambda>\tilde{\lambda}$ a solution $u_\lambda \in \huz\cap\lio$ of~\eqref{eq:Problema_general} such that
\[
\underline u_\lambda(x)\leq u_\lambda(x) \leq \overline u_\lambda(x) \textrm{ a.e. }x\in \Omega.
\]
We point out that $\|u_\lambda\|_{L^\infty(\Omega)}\in(\alpha,\beta)$ because $\|\underline u_\lambda\|_{L^\infty(\Omega)}\in(\alpha,\beta)$ and $\|\overline{u}_\lambda\|_{L^\infty(\Omega)} = \beta$. Thus, the set $\Lambda_{\alpha\beta}$ is not empty and, in fact, $(\tilde \lambda, +\infty) \subseteq \Lambda_{\alpha\beta}$.

{\bf Third step}. We prove that $\Lambda_{\alpha\beta}$ is closed. 

Let $\lambda_n$ be a convergent sequence in $\Lambda_{\alpha\beta}$ and denote $\lambda=\lim_{n\to \infty} \lambda_n$. Let us take $u_n \in H_0^1(\Omega)\cap L^\infty(\Omega)$ with $\alpha< \|u_n\|_{L^\infty(\Omega)}<\beta$ satisfying 
\begin{equation*}
\begin{cases}
-\Delta u_n  = \lambda_n f(u_n) & \mbox{in} \; \Omega,\\
u_n = 0 & \mbox{on} \; \partial \Omega.\\
\end{cases} 
\end{equation*}
Observe that the sequence $\{\lambda_n f(u_n)\}$ is bounded in $L^\infty(\Omega)$. In addition, applying~\cite[Theorem 6.1]{L-U} we deduce that the sequence $u_n$ is bounded in $\mathcal{C}^{0,\gamma}(\overline\Omega)$. Consequently, Ascoli-Arzel\'{a} Theorem assures that $u_n$ possesses a subsequence converging strongly in $\mathcal{C}(\overline{\Omega})$ to $u\in L^\infty(\Omega)$ with $\alpha\leq  \|u\|_{L^\infty(\Omega)}\leq \beta$.

Moreover, taking $u_n$ as test function and using that $f$ is continuous we get
that $u_n$ is bounded in $H_0^1(\Omega)$. This implies, in particular, that $u\in H_0^1(\Omega)$ and, up to a subsequence, $u_n\rightarrow u$ weakly in $H_0^1(\Omega)$. We can pass to the limit and it follows that 
\begin{equation*}
\begin{cases}
\displaystyle -\Delta u  = \lambda f(u) & \mbox{in} \; \Omega,\\
u = 0 & \mbox{on} \; \partial \Omega.\\
\end{cases} 
\end{equation*}
Since $u$ solves \eqref{eq:Problema_general} the strong maximum principle allows us to ensure that $\alpha\ne \|u\|_{L^\infty(\Omega)}\ne \beta$ (see Lemma \ref{lema:no_existencia}). Therefore, we have proved that $\lambda\in \Lambda_{\alpha\beta}$, i.e. $\Lambda_{\alpha\beta}$ is closed which concludes the third step.

{\bf Fourth step.} We prove that $\Lambda_{\alpha\beta}=[\lambda_{\alpha\beta}, +\infty)$ for some $\lambda_{\alpha\beta}>0$. 

Since $\Lambda_{\alpha\beta}$ is nonempty and closed, we can take $\lambda_{\alpha\beta} = \min \Lambda_{\alpha\beta}$. First, we observe that $\lambda_{\alpha\beta} \neq 0$ because the unique solution to~\eqref{eq:Problema_general} with $\lambda=0$ is the trivial one. Finally, since $f$ is nonnegative, every solution to $(P_{\lambda_{\alpha\beta}})$ is a subsolution to~\eqref{eq:Problema_general} with $\lambda>\lambda_{\alpha\beta}$, so we can argue as in the second step to deduce that $[\lambda_{\alpha\beta}, +\infty) \subseteq \Lambda_{\alpha\beta}$ and thus $\Lambda_{\alpha\beta} = [\lambda_{\alpha\beta}, +\infty)$.
\end{proof}

\begin{remark}
 Let us remark some facts related with the previous result and its proof:
\begin{enumerate}\label{rem:ramarriba+multiplicidad}
 \item Arguing as in the first step with $\tilde f(s)=0$ for $s\in \mathbb{R}\setminus (\beta-\varepsilon,\beta)$ and $0<\tilde f(s)\leq f(s)$  in $(\beta-\varepsilon,\beta)$ we obtain the existence of $\lambda_\varepsilon\to +\infty$, as $\varepsilon\to 0$ and a subsolution $\underline u_\varepsilon$ to $(P_{\lambda_\varepsilon})$ such that $\|\underline u_\varepsilon\|_{L^\infty(\Omega)}\to \beta$ as $\varepsilon\to 0$. Then, as in the second step, we obtain solutions $u_\varepsilon$ with the same property, i.e. $\|u_\varepsilon\|_{L^\infty(\Omega)}\to \beta$ as $\varepsilon\to 0$. 
 
 \item The second step can also be deduced as in~\cite{Hess} using Lemma~\ref{lema:no_existencia} instead of Lemma~\ref{lem:AmbHess}. As a consequence, a multiplicity result in $\Lambda_{\alpha\beta}$ for large $\lambda$ is also deduced. Indeed, the first solution in $\Lambda_{\alpha\beta}$ is found for large $\lambda$ by minimization of a truncated functional which is either not isolated (leading to multiplicity) or has Leray-Schauder index equal to one. In this case, some degree computations lead to the existence of a second solution in $\Lambda_{\alpha\beta}$.
\end{enumerate}
\end{remark}

Now we apply the Leray-Schauder degree to obtain the existence of an unbounded continuum of solutions of \eqref{eq:Problema_general}. We can consider the operator $K\colon
\mathbb{R}\times \mathcal{C}(\overline{\Omega}) \to \mathcal{C}(\overline{\Omega})$ by defining, for every $\lambda\in\mathbb{R}$ and for every $w\in  \mathcal{C}(\overline{\Omega})$, $K(\lambda ,w)=\lambda^+(-\Delta)^{-1}(f(w))$, i.e. the unique solution $u$ in $H^1_0(\Omega)\cap \mathcal{C}(\overline{\Omega})$ of the problem
\[
\begin{cases}
\dys-\Delta u=  \lambda^+ f(w)  & \text{in}\ \Omega,\\[0.7 ex]
u=0 &\text{on}\ \partial\Omega.
\end{cases}
\]
It is well-known that $K$ is well defined. Furthermore, by using the properties of $-\Delta$, we can ensure that $K$ is compact. 

%
%
%

In the following, for some $\gamma\in (0,1)$, we denote $O_{\alpha\beta}=\{u\in \mathcal{C}^{1,\gamma}_0(\overline \Omega): \alpha<\|u\|_{L^\infty(\Omega)}<\beta\}$. Now we prove Theorem~\ref{teor:continuo}.

\begin{proof}[{\bf Proof of Theorem~\ref{teor:continuo}}]

The first part of the proof is obtained directly from Lemma~\ref{lema:existencia}. Let us now assume~\ref{eq:hip_f_creciente} for $\sigma=\beta$ to prove the existence of the unbounded continuum $\Sigma \subset \mathcal{S}_{\alpha\beta}$.

Let $v$ be the maximal solution of $(P_{\lambda_{\alpha\beta}})$ in $O_{\alpha\beta}$. This means that if $u$ solves $(P_{\lambda_{\alpha\beta}})$ with $\alpha<\|u\|_{L^\infty(\Omega)} <\beta$, then $u\leq v$. The existence of this maximal solution is due to the fact that thanks to~\ref{eq:hip_f_creciente} we can use an iterative scheme in which we take as a starting point the constant function $\beta$.

Observe that since $f$ is nonnegative, $\underline u = v$ is a subsolution of $(P_\lambda)$ which is not a solution for every $\lambda>\lambda_{\alpha\beta}$. Moreover, given $\overline{\lambda}> \lambda_{\alpha\beta}$, we consider $\overline u$ satisfying 
\begin{equation}
\label{eq:Dem_teor_continuo_1}
\begin{cases}
\displaystyle -\Delta \overline u  + \overline{\lambda} M \overline u = \overline{\lambda} M\beta & \mbox{in} \; \Omega,\\
\overline u = 0 & \mbox{on} \; \partial \Omega.\\
\end{cases} 
\end{equation}
We claim that $\overline u$ is a supersolution of $(P_\lambda)$ which is not a solution for every $\lambda\leq \overline{\lambda}$ and $\underline u\leq \overline u\leq \beta$. Indeed, using $(\overline u-\beta)^+$ as test function in \eqref{eq:Dem_teor_continuo_1} we prove that $\overline u\leq \beta$. Then, using that $\overline{\lambda} M \beta =\overline{\lambda} (f(\beta)+M\beta) \geq \lambda (f(\overline u)+M\overline u)$, we obtain that $\overline u$ is a supersolution of $(P_\lambda)$ which is not a solution for every $\lambda\leq \overline{\lambda}$. Moreover, using that $\overline{\lambda} M \beta = \overline{\lambda} (f(\beta)+M\beta) \geq \lambda (f(\underline u)+M\underline u)$ and the comparison principle, we find that $\underline u\leq \overline u$.

Now we denote 
\[
O_1= \left\{u\in \mathcal{C}^{1,\gamma}_0(\overline \Omega): \underline u<u<\overline u \textrm{ in }\Omega, \frac{\partial \overline u}{\partial n_e}<\frac{\partial u}{\partial n_e}<\frac{\partial \underline u}{\partial n_e}  \textrm{ on }\partial \Omega \right\}.
\]
Given $\lambda_1< \lambda_{\alpha\beta}$ we also define $T\colon \left[\lambda_1,\overline{\lambda} \right] \times \mathcal{C}^{1,\gamma}_0(\overline \Omega)\to \mathcal{C}^{1,\gamma}_0(\overline \Omega)$ given by $T(\lambda,u)=u-K(\lambda,u)$.

The strong maximum principle (Lemma~\ref{lema:no_existencia}) allows us to ensure that $T(\lambda,u)\ne 0$ for every $\lambda\in \left[\lambda_1, \overline{\lambda} \right]$ and $u\in \partial O_{\alpha\beta}$. Moreover, it also implies that $T(\lambda,u)\ne 0$ for every $\lambda\in \left(\lambda_{\alpha\beta},\overline{\lambda} \right]$ and $u\in \partial O_1$. In particular, it is well defined Leray-Schauder topological degree and satisfies
\begin{enumerate}[1)]
\item $\mbox{\rm deg}(T(\lambda,\cdot),O_{\alpha\beta},0)=\mbox{\rm deg}(T(\lambda_1,\cdot),O_{\alpha\beta},0)=0$, for every $\lambda\in [\lambda_1,\overline{\lambda}]$.
\item Arguing as in \cite{Lazer-McKenna} (see also \cite{A-C-2}) we claim now that there exists $r \geq 0$, such that, for every $\lambda\in \left(\lambda_{\alpha\beta},\overline{\lambda} \right]$, one has $\mbox{\rm deg}(T(\lambda,\cdot),{{O}_1}\cap B_r(0),0)=1$. Indeed, we define the truncated function $g$ in
the following way:
\begin{equation*}
 g(x,s)= \left\{
\begin{array}{lcl}
f(\underline u(x))+M\underline u(x) &\mbox{ if }& s\leq \underline u(x), \\
f(s)+M s &\mbox{ if }& \underline u(x)<s< \overline u(x), \\
f(\overline u(x))+M \overline u(x) &\mbox{ if }& s\geq \overline u(x),
\end{array}
\right.
\end{equation*}
where $M>0$ is given by~\ref{eq:hip_f_creciente}. Observe that $g(x,s)$ is bounded and non-decreasing in the $s$ variable. Consider now for every $h\in H^{-1}(\Omega)$ the unique solution $u=T_\lambda h$ of the problem
\begin{equation*}
\left\{
\begin{array}{rclr}
-\Delta u + \lambda M u &
= & h \quad & x\in \Omega,
\\ u & = & 0 & x\in \partial \Omega.
\end{array}
\right.
\end{equation*}
This define an operator $T_\lambda:H^{-1}(\Omega) \to H_0^1(\Omega)$. Since $g(x,s)$ is bounded, by \cite[Theorem~9.15]{Gil-Trud} and the Morrey Theorem, we know that $T_\lambda(\lambda g(x,u))$ is bounded in $\mathcal{C}^1(\overline \Omega)$. Let
\[
r_0=\sup\{\|T_\lambda(\lambda g(x,v))\|_{\mathcal{C}^1(\overline \Omega)}: \,
v\in \mathcal{C}^1(\overline \Omega)\}
\]
and choose $r>\max\{\beta,r_0\}$. By the definition of $g$ we have that $g(x, \underline u(x)) \leq g(x,v(x))\leq g(x, \overline u(x))$ for every $v \in \mathcal{C}^1_0(\overline\Omega)$ and consequently, by using the strong maximum principle, we obtain that
\[
\{T_\lambda(\lambda g(x,v(x))): v\in  \mathcal{C}^1_0(\overline\Omega)\} \subset
{O_1}\cap B_r(0).
\]

Now, pick $\psi \in {O_1}\cap B_r(0)$ and consider the compact homotopy, $H(s,u)=s T_\lambda(\lambda g(x,u))+(1-s)\psi$, $0\leq s\leq 1$. Since ${O}_1\cap B_r(0)$ is a convex set 
we have that $u \ne H(s,u)$ for all $u\in \partial {{O}_1}\cap B_r(0)$ and $s \in [0,1]$. Therefore
\[
\mbox{\rm deg}(I-T_\lambda(\lambda g(x,\cdot)),{O_1}\cap B_r(0),0)=\mbox{\rm deg}(I-\psi,{O_1}\cap B_r(0),0)=1.
\]

Noting now that $g(x,v(x))=f(v(x))+Mv(x)$ for every $v \in \overline{{O}_1\cap B_r(0)}$, we yield to
\[
\mbox{\rm deg}(T(\lambda,\cdot),{{O}_1}\cap B_r(0),0)=\mbox{\rm deg}(I-T_\lambda(\lambda g(x,\cdot)),{{O}_1}\cap B_r(0),0) =1,
\]
proving the claim. 
\end{enumerate}

By using \cite[Lemma 8]{A-C-2}  (with $a$ and $b$ interchanged) we deduce that there exist a continuum $\Sigma_{\overline{\lambda}}$ of solutions such that $\Sigma_{\overline{\lambda}} \cap (\{ \overline{\lambda} \} \times {O}_1\cap B_r(0))\ne \emptyset$ and $\Sigma_{\overline{\lambda}} \cap \{\overline{\lambda}\}\times (O_{\alpha\beta}\setminus \overline{{O}_1\cap B_r(0)} )\ne \emptyset$.

Moreover, the strong maximum principle assures that there is no solution of the problem on $\partial O_1\cap B_r(0)$ for every $\lambda\in (\lambda_{\alpha\beta},\overline{\lambda}]$. Thus, due to the connection of $\Sigma_{\overline{\lambda}}$ we deduce that $\Sigma_{\overline{\lambda}}$ intersects $\{\lambda\}\times {\partial O_1\cap B_r(0)}$ with $\lambda\in [\lambda_{\alpha\beta},\overline{\lambda}]$ if and only if $\lambda=\lambda_{\alpha\beta}$. This implies that there exists $u\geq \underline{u}$ such that $(\lambda_{\alpha\beta}, u)\in \Sigma_{\overline{\lambda}}\cap \{\lambda_{\alpha\beta} \}\times \partial O_1$. Since $\underline{u}$ is maximal we deduce that $u=\underline{u}$.

It is enough to define $\Sigma =\displaystyle\bigcup_{\overline{\lambda} > \lambda_{\alpha\beta}} \Sigma_{\overline{\lambda}}$. Observe that, since $(\lambda_{\alpha\beta}, \underline{u})\in \Sigma_{\overline{\lambda}}$ for every $\lambda>\lambda_{\alpha\beta}$, the set $\Sigma$ is connected. 
\end{proof}

\section{Behavior of the solutions when the nonlinearity $f$ has a monotone sequence of positive zeros}
\label{sec:sucesion_de_ceros}

In this section we get a full picture of the bifurcation diagram of~\eqref{eq:Problema_general} when the function $f$ has a monotone sequence of positive zeros.  

In order to do that, we assume that $f$ is a nonnegative function having a monotone sequence of positive zeros. We denote this sequence by $\{\alpha_n\}$ and we also suppose that \ref{eq:hip_f_creciente_3} is satisfied for $\sigma=\alpha_n$ for every $n\in\na$.


We define
\begin{align}
    \label{eq:def_Lambda_n}
    \Lambda_n := & \left\{\lambda\geq 0 : \exists u\in\huz \cap L^\infty(\Omega) \textrm{ solution of~\eqref{eq:Problema_general}}\right.
    \\ \nonumber
  & \left.  \textrm{ with } \min\{\alpha_n, \alpha_{n+1}\} < \|u\|_{L^\infty(\Omega)} < \max\{\alpha_n, \alpha_{n+1}\}\right\}
\end{align}
and
\begin{equation}
    \label{eq:def_lambda_n}
    \lambda_n := \inf \Lambda_n
\end{equation}
for every $n\in\na$. Lemma~\ref{lema:existencia} implies that  $\lambda_n>0$ and $\Lambda_n=[\lambda_n,+\infty)$. 

Our main goal is to study the behavior of the sequence $\{\lambda_n\}$ when $n$ goes to infinity to get a full picture of the bifurcation diagram of~\eqref{eq:Problema_general}. When $\alpha_n\to \alpha$ for some $\alpha>0$ it is easy to prove that $\lambda_n\to +\infty$, as we show in the following result.

\begin{proposition} \label{prop:acum_pos}
    Let $f$ be a nonnegative continuous function having a monotone sequence $\{\alpha_n\}$ of positive zeros converging to some $\alpha\in (0,+\infty)$. Assume that $f$ verifies~\ref{eq:hip_f_creciente_3} for $\sigma=\alpha$ and $\sigma=\alpha_n$ for every $n\in \mathbb{N}$. Then the sequence $\{\lambda_n\}$ defined in~\eqref{eq:def_lambda_n} diverges.
\end{proposition}

\begin{proof}
Arguing by contradiction, we suppose that $\{\lambda_n\}$ does not diverge. Then, there exists a subsequence of $\{\lambda_n\}$, still denoted by $\{\lambda_n\}$ for simplicity, converging to some $\lambda\geq 0$. By Lemma~\ref{lema:existencia}, we can take for every $n\in\na$ a solution $u_n$ of $(P_{\lambda_n})$ with 
\begin{equation}
    \label{eq:Dem_prop_monotona_1}
    \min\{\alpha_n, \alpha_{n+1}\} < \|u_n\|_{L^\infty(\Omega)} < \max\{\alpha_n, \alpha_{n+1}\}.
\end{equation}
Since $\{\lambda_n f(u_n)\}$ is bounded in $\lio$, we can argue as in the third step of the proof of Lemma~\ref{lema:existencia} to show that $u_n\to u$ in $\mathcal{C}(\overline{\Omega})$ for some $u\in\huz\cap \mathcal{C}(\overline{\Omega})$ solution of~\eqref{eq:Problema_general}. Passing to the limit in~\eqref{eq:Dem_prop_monotona_1}, we get $\|u\|_{L^\infty(\Omega)} = \alpha$, but this contradicts Lemma~\ref{lema:no_existencia} since $f(\alpha)=0$.
\end{proof}

\begin{figure}[ht]
\label{fig:ceros_a_alfa_pos}
\centering
  \includegraphics[width=5cm]{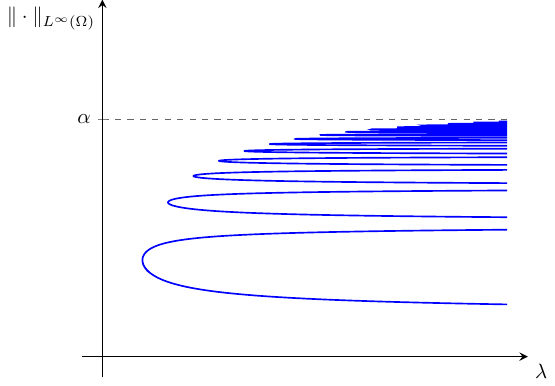}
\caption{Sketch of the solution set of problem~\eqref{eq:Problema_general} on the basis of Proposition~\ref{prop:acum_pos}.}
\end{figure}

When $\alpha_n\to 0$ or $\alpha_n\to +\infty$ one could expect the situation to be similar. However, in such cases the behavior of $\{\lambda_n\}$ depends strongly on $f$. In the following subsections, we study each of these cases separately. We will show, considering a particular family of functions $f$, that $\{\lambda_n\}$ can behave in several ways, some of them implying the existence of infinite bifurcation points from zero (when $\alpha_n\to 0$) or from infinity (when $\alpha_n\to+\infty$) which are not branching points in the sense of the next definition.


\begin{definition}

\begin{enumerate}
\item  A nonnegative number $\overline{\lambda}$ is said to be a \textit{bifurcation point from zero} to problem~\eqref{eq:Problema_general} whenever $(\overline{\lambda},0)\in \mathcal{B}_0$. 
If there also exists a connected and closed subset $\Sigma_0$ of $\overline{\mathcal{S}\setminus \mathcal{T}}$ such that $(\overline{\lambda},0)\in \Sigma_0$, then we say that $\overline{\lambda}$ is a \textit{branching point from zero}.

\item  A nonnegative number $\overline{\lambda}$ is said to be a \textit{bifurcation point from infinity} to problem~\eqref{eq:Problema_general} whenever $(\overline{\lambda},0)\in \mathcal{B}_\infty$. 
If there also exists a connected and closed subset $\Sigma_\infty$ of $\overline{\mathcal{S}}_\infty$ such that $(\overline{\lambda}, 0)\in \Sigma_\infty$
is connected and closed, then we say that $\overline{\lambda}$ is a \textit{branching point from infinity}.
\end{enumerate}
\end{definition}

The next result shows that whenever $\{\lambda_n\}$ does not diverge then infinitely many bifurcation points arise (actually an unbounded interval).
\begin{proposition}
\label{prop:bifurcacion}
Let $f$ be a nonnegative continuous function having a monotone sequence $\{\alpha_n\}$ of positive zeros converging either to $0$ or to $+\infty$ and verifying~\ref{eq:hip_f_creciente_3} for $\sigma=\alpha_n$ for every $n\in \mathbb{N}$. If the sequence $\{\lambda_n\}$ defined in~\eqref{eq:def_lambda_n} has an accumulation point $\lambda_0\geq 0$, then every $\lambda>\lambda_0$ is
\begin{enumerate}[i)]
    \item a bifurcation point from infinity but not a branching point if $\alpha_n\to+\infty$,
    \item a bifurcation point from zero but not a branching point if $\alpha_n\to 0$.
\end{enumerate}
\end{proposition}

\begin{proof}
Due to the similarities between the proof of both points, we give the proof of the first one and we indicate in parentheses the small changes related to the second one.

Let $\{\lambda_{\sigma(n)}\}$ be a subsequence convergent to $\lambda_0\geq 0$. Fix some $\lambda >\lambda_0$. Then there exists some $n_0\in\na$ such that $\lambda_{\sigma(n)}<\lambda$ for every $n\geq n_0$. Thus, $\lambda\in \Lambda_{\sigma(n)}$ for every $n\geq n_0$, where $\Lambda_n$ is defined in~\eqref{eq:def_Lambda_n}. The very definition of $\Lambda_n$ implies that $\lambda$ is a bifurcation point from infinity (resp. from zero).

On the other hand, by Lemma~\ref{lema:no_existencia} no $\lambda\geq 0$ can be a branching point from infinity (resp. from zero) since there are no positive solutions of~\eqref{eq:Problema_general} with $\lio$-norm equal to any $\alpha_n$.
\end{proof}

In order to obtain sufficient conditions for the boundedness of the sequence $\{\lambda_n\}$ it is important to note that, since the constants $\alpha_n$ are always supersolutions, those conditions depend on the existence of a convenient weak subsolution of~\eqref{eq:Problema_general}. This is the main goal of the next result whose proof follows the lines of~\cite[Lemma~3]{GM-I}.

\begin{remark}\label{rem:tecnica}
Assume that $f_1$ is a nonnegative function and $F_1(s)=\int_0^s f_1(t)\, \mathrm{d}t$. Then, given $\nu >0$ such that $f(\nu)\ne 0$ we have that 
\begin{equation*}
\int_0^\nu \frac{\mathrm{d}s}{\sqrt{F_1(\nu)-F_1(s)}} <+\infty.
\end{equation*}
Indeed, this is a consequence of having that $\displaystyle \lim_{s\to \nu^-} \frac{\nu-s}{F_1(\nu)-F_1(s)}=\frac{1}{f_1(\nu)}$. In the following we will use the notation 
\begin{equation}
    \label{eq:def_g}
    \hat f_1 (\nu) = \int_0^\nu \frac{\mathrm{d}s}{\sqrt{F_1(\nu)-F_1(s)}},\ \forall \nu\in[0,+\infty[ \textrm{ with } f_1(\nu)\ne 0.
\end{equation}
\end{remark}

%
\begin{lemma}
\label{lema:radial}
Assume that $f_1$ is a nonnegative function such that $f_1(s)\leq f(s)$, $s\geq 0$. For every $\nu>0$ with $f_1(\nu)\ne 0$ there exists $\lambda >0$ for which \eqref{eq:Problema_general} admits a nonnegative weak subsolution $\underline{u}$ such that $\|\underline{u}\|_{L^\infty(\Omega)} = \nu$.
\end{lemma}


\begin{proof}
We consider $x_0\in \Omega$ and $R$ small enough such that $A=\{x\in\re^N: \frac{R}{2}<|x-x_0|<R\}\subset \subset \Omega$ and take $\lambda$ such that
\begin{equation}
\label{eq:Lema_radial_igualdad}
\hat f_1(\nu) = R\sqrt{\lambda} \, \frac{2^{N-2}-1}{2^{N-\frac12}},
\end{equation}
where $\hat f_1$ is defined in Remark~\ref{rem:tecnica} and we have used that $f_1(\nu)\ne 0$.

We claim that there exists a radially symmetric subsolution of the problem~\eqref{eq:Problema_general}, with $\Omega$ replaced by $A$ which, extended by zero to $\Omega$, constitutes a subsolution of~\eqref{eq:Problema_general}. Indeed, this radially symmetric subsolution of~\eqref{eq:Problema_general} in $A$, denoted by $w(r)$, where $r=|x-x_0|$, will satisfy
\begin{equation}
\label{eq:Dem_Lema_radial_1}
    \begin{cases}
   -w'' - \frac{N-1}{r}w' \leq \lambda f_1(w) &\text{ in } \frac{R}{2}<r<R,\\
   w\left(\frac{R}{2}\right) = w(R) = 0. &
    \end{cases}
\end{equation}

Now, in order to eliminate the first order term, we perform the change of variables $w(r) = z(s)$ with $s=\frac{r^{2-N}}{N-2}$  if $N\geq 3$ or $s=\ln (r)$ if $N=2$. Then, the above problem is equivalent to
\begin{equation*}
    \begin{cases}
    -z'' \leq \lambda g_1(s)f_1(z) &\text{ in } a<s<b,\\
    z(a) = z(b) = 0 &
    \end{cases}
\end{equation*}
with $a=\frac{1}{R^{N-2}(N-2)}$, $b=\frac{2^{N-2}}{R^{N-2}(N-2)}$ if $N\geq 3$; $a=\ln\left(\frac{R}{2}\right)$, $b=\ln(R)$ if $N=2$ and $g_1(s)=r^{2(N-1)}$ in both cases (i.e., if $N\geq 2$). Since $f_1$ is nonnegative, to our purpose it suffices to obtain a solution of the problem
\begin{equation*}
    \begin{cases}
    -z'' = \lambda \left(\frac{R}{2}\right)^{2(N-1)} f_1(z) &\text{ in } a<s<b,\\
    z(a) = z(b) = 0. &
    \end{cases}
\end{equation*}

This problem can be solved using standard techniques. We can find a positive solution, which is symmetric with respect to the midpoint of the interval $[a,b]$ and which is given implicitly by
\[
\int_0^{z(s)} \frac{\mathrm{d}z}{\sqrt{F_1(\nu)-F_1(z)}} = (2\lambda)^{\frac{1}{2}} \left(\frac{R}{2}\right)^{N-1} (s-a), \qquad a\leq s\leq\frac{a+b}{2},
\]
where $\nu=\max_{s\in[a,b]} z(s)$ is a solution of the equation
\[
\int_0^\nu \frac{\mathrm{d}s}{\sqrt{F_1(\nu)-F_1(s)}} = (2\lambda)^\frac{1}{2} \left(\frac{R}{2}\right)^{N-1} \frac{b-a}{2}.
\]
See that $\nu$ and $\lambda$ satisfy this relation by \eqref{eq:Lema_radial_igualdad}.

Undoing the change of variables, we obtain that $w(r)$ verifies~\eqref{eq:Dem_Lema_radial_1}. Finally, we define
\begin{equation*}
    \underline{u}(x) =
    \begin{cases}
    w(x) & \text{ if } x\in A,\\
    0 & \text{ if } x\in\Omega\backslash A.
    \end{cases}
\end{equation*}
and in this way we obtain a subsolution of~\eqref{eq:Problema_general} satisfying that $\|\underline{u}\|_{L^\infty(\Omega)} = \nu$.
\end{proof}


%


\subsection{The problem \eqref{eq:Problema_general} with $f$ having a divergent sequence of zeros.}\hfill
\label{sec:seno_u}


In this subsection we consider a wide class of particular cases where the nonlinearity is given by $f(t)=t^r(1+\sin t)g(t)$ for some $r\geq 0$ and $g$ satisfying \eqref{hipg}, i.e., we study the problem
\begin{equation*}
    \tag{$Q_\lambda$}
    \begin{cases}
   -\Delta u = \lambda u^r(1+\sin u) g(u) &\text{ in } \Omega,\\
   u=0 &\text{ on } \partial\Omega.
    \end{cases}
\end{equation*}
 In this case the sequence $\{\alpha_n\}$ of positive zeros of $f$ diverges and can be explicitly determined by
\begin{equation}
    \label{eq:def_alpha_n_seno_u}
    \alpha_n = -\frac{\pi}{2} + 2\pi n,\ \forall n\in\na.
\end{equation}
We observe that $f$ satisfies the assumption~\ref{eq:hip_f_creciente_3} for $\sigma=\alpha_n$ for every $n\in \mathbb{N}$.

The next result provides a lower bound for $\lambda_n$ defined in~\eqref{eq:def_lambda_n} for the problem $(Q_\lambda)$ in terms of the sequence $\{\alpha_n\}$.


\begin{proposition}
\label{prop:cota_inferior_seno_u}
Assume \eqref{hipg}, let $\alpha_n$ be given by~\eqref{eq:def_alpha_n_seno_u} and let $\{\lambda_n\}$ be the sequence defined in~\eqref{eq:def_lambda_n} for the problem~\eqref{eq:Problema_seno_u}. Then there exists some $C>0$ independent from $n$ such that
\[
C\frac{\alpha_n}{\alpha_{n+1}^r} \leq \lambda_n,\ \forall n\in\na.
\]
\end{proposition}
\begin{proof}
By the definition of $\lambda_n$, we can consider, for every $n\in\na$, a solution $u_n\in\huz\cap \mathcal{C}\left(\overline{\Omega}\right)$ of $(Q_{\lambda_n})$ satisfying that 
\begin{equation}
    \label{eq:Dem_cota_inferior_seno_u_1}
    \alpha_n < \|u_n\|_{L^\infty(\Omega)} < \alpha_{n+1}.
\end{equation}

Taking $G_k(u_n) = \max\{u_n-k, \min\{u_n+k,0\}\} \in\huz$ with $k>0$ as test function in $(Q_{\lambda_n})$ and using~\eqref{eq:Dem_cota_inferior_seno_u_1}, we obtain that
\begin{align*}
 \into |\nabla G_k(u_n)|^2 = & \lambda_n \into u_n^r (1+\sin u_n)g(u_n) G_k(u_n) 
 \\
 \leq & 2\sup_{s\geq 0} g(s)\, \alpha_{n+1}^r\lambda_n \into G_k(u_n).
\end{align*}
By Stampacchia Theorem (see~\cite{Stamp}), this implies the existence of some $C(N,\Omega)>0$ such that
\[
\|u_n\|_{L^\infty(\Omega)} \leq C\alpha_{n+1}^r \lambda_n,\ \forall n\in\na.
\]
This inequality combined with~\eqref{eq:Dem_cota_inferior_seno_u_1} allow us to finish the proof.
\end{proof}

We now prove a technical lemma which, combined with Lemma~\ref{lema:radial}, will provide us the upper bound for $\{\lambda_n\}$ from a certain $n_0\in\na$. It is related with the function $\hat f$ defined in~\eqref{eq:def_g} and with the points $\nu_n=\alpha_n+ \pi$, i.e. 
\begin{equation}
    \label{eq:def_M_n_seno_u}
    \nu_n = \alpha_n+\pi = \frac{\pi}{2} + 2\pi n,\ \forall n\in\na.
\end{equation}
Observe that $1+\sin \nu_n=2$ and $\alpha_{n}<\nu_n<\alpha_{n+1}$ for every $n\in\na$. 

\begin{lemma}
\label{lema:cota_superior_seno_u} Let \eqref{hipg} be satisfied and denote $\gamma = \inf_{s\geq 0}g(s)>0$. Let $f_1(t) = \gamma\, t^r(1+\sin t)$ with $r\geq 0$, $\nu_n$ defined by~\eqref{eq:def_M_n_seno_u} and let $\hat f_1$ be defined by~\eqref{eq:def_g}. Then there exists $n_0\in\na$ and $C>0$ independent of $n$ such that
\[
\hat f_1(\nu_n) \leq C \nu_n^\frac{1-r}{2},\ \forall n\geq n_0.
\]
\end{lemma}
\begin{proof}
Along the proof we will denote by $C$ several positive constants independent of $n$ whose value may change from line to line and, sometimes, on the same line.

We claim that there exists $n_0\in\na$ and $C>0$ such that
\begin{equation}
\label{eq:Dem_lema_cota_sup_seno_u_6}
h_n(s)\equiv \frac{\int_s^{\nu_n} f_2(t) \, \mathrm{d}t}{2\nu_n^r(\nu_n-s)} \geq C,\ \forall s\in [0,\nu_n[,\ \forall n\geq n_0,
\end{equation}
where $f_2(t) \equiv \gamma^{-1}f_1(t) = t^r(1+\sin t)$.

Assuming that the claim is true, our assertion immediately follows. Indeed, in that case:
\begin{align*}
    \hat f_1(\nu_n) &= \int_0^{\nu_n} \frac{\mathrm{d}s}{\sqrt{F_1(\nu_n)-F_1(s)}} 
    = \int_0^{\nu_n} \frac{\mathrm{d}s}{\sqrt{\int_s^{\nu_n} f_1(t) \, \mathrm{d}t}} \\
    &= \int_0^{\nu_n} \frac{\mathrm{d}s}{\sqrt{\gamma \int_s^{\nu_n} f_2(t) \, \mathrm{d}t}}
    \leq \int_0^{\nu_n} \frac{\mathrm{d}s}{\sqrt{ 2\gamma C\nu_n^r (\nu_n-s)}} \\
    &= \frac{1}{\sqrt{2\gamma C\nu_n^r}} \left[-2\sqrt{\nu_n-s} \right]_{s=0}^{s=\nu_n} 
    = \frac{\sqrt{2}}{\sqrt{\gamma C}} \nu_n^{\frac{1-r}{2}},\ \forall n\geq n_0,
\end{align*}
and the proof finishes. So to conclude it remains only to prove the claim~\eqref{eq:Dem_lema_cota_sup_seno_u_6}. 
%

First of all, using L'HÃŽpital's rule we obtain that
\begin{equation}
    \label{eq:Dem_lema_cota_sup_seno_u_1}
    \lim_{s\to \nu_n} h_n(s) = \lim_{s\to \nu_n} \frac{-f_2(s)}{-2\nu_n^r} = \frac{f_2(\nu_n)}{2\nu_n^r} = 1,\ \forall n\in\na.
\end{equation}

Now we prove that 
\begin{equation}
    \label{eq:Dem_lema_cota_sup_seno_u_2}
    h_n(s)\geq C,\ \forall s\in[0, \nu_n-c_r],\ \forall n\in\na,
\end{equation} 
for some $c_r>0$ independent on $n\in \mathbb{N}$. 
Taking into account that
\[
\int_s^{\nu_n} f_2(t) \, \mathrm{d}t = \frac{1}{r+1}(\nu_n^{r+1} - s^{r+1}) - \nu_n^r \cos \nu_n + s^r\cos s + r\int_s^{\nu_n} t^{r-1} \cos t\, \mathrm{d}t
\]
and that since $\cos \nu_n =0$ we have
\[
\left|- \nu_n^r \cos \nu_n + s^r\cos s + r\int_s^{\nu_n} t^{r-1} \cos t\, \mathrm{d}t \right| \leq s^r|\cos s| + \nu_n^r - s^r \leq \nu_n^r.
\]
Therefore we deduce that
\[
\int_s^{\nu_n} f_2(t) \, \mathrm{d}t \geq \frac{1}{r+1}(\nu_n^{r+1} - s^{r+1}) - \nu_n^r.
\]
Thus, we obtain that
\begin{align*}
h_n(s) &\geq  \frac{1}{2\nu_n^r(\nu_n-s)} \left[\frac{1}{r+1}(\nu_n^{r+1} - s^{r+1}) - \nu_n^r\right]\\
&= \frac{1}{2(r+1)} \left[1+ \frac{s(\nu_n^r-s^r)}{\nu_n^r(\nu_n-s)}\right] - \frac{1}{2(\nu_n-s)} \geq \frac{1}{2(r+1)} - \frac{1}{2(\nu_n-s)}.
\end{align*}
In particular, we get that $h_n(s)\geq C>0$ whenever 
\[
s< \nu_n - \frac{1}{\frac{1}{r+1}-2C},
\]
being $C$ a positive constant small enough. Thus, we can choose $c_r = r+1+\varepsilon$ with $\varepsilon>0$ small and~\eqref{eq:Dem_lema_cota_sup_seno_u_2} follows.

It remains only to show that 
\begin{equation}
    \label{eq:Dem_lema_cota_sup_seno_u_5}
    h_n(s)\geq C,\ \forall s\in[\nu_n-c_r, \nu_n[,\ \forall n\geq n_0,
\end{equation}
for some $n_0\in \mathbb{N}$. Observe that if $h_n$ does not have any interior minimum point in that interval, then this is a direct consequence  of~\eqref{eq:Dem_lema_cota_sup_seno_u_1} and~\eqref{eq:Dem_lema_cota_sup_seno_u_2}. We assume now that $h_n$ has some minimum critical point $s_0\in ]\nu_n-c_r, \nu_n[$ (observe that $s_0$ depends on $n$). In what follows we are going to prove that $h_n(s_0)\geq C$ for $n$ large, and this will imply that $h_n(s) \geq C$ for $s\in [\nu_n-c_r, \nu_n[$ and $n$ large.

We first observe that $h_n'(s_0)=0$, i.e. we have that
\[
\int_{s_0}^{\nu_n} f_2(t)\, \mathrm{d}t = f_2(s_0) (\nu_n-s_0).
\]
Thus, $s_0$ satisfies that
\begin{equation*}
    h_n(s_0) = \frac{f_2(s_0)}{2\nu_n^r} = \frac{s_0^r (1+\sin s_0)}{2\nu_n^r} \geq \frac{1+\sin s_0}{2} \frac{(\nu_n-c_r)^r}{\nu_n^r} \geq C \frac{1+\sin s_0}{2},\ \forall n\geq n_0,
\end{equation*}
for some $n_0\in\na$ large, since $\frac{\nu_n-c_r}{\nu_n} \to 1$.

Let us finally check that $1+\sin s_0$ is bounded below by a positive constant. This is clear if the set of zeros of the function $1+\sin s$ does not intersect the interval $[\nu_n-c_r, \nu_n[$. Otherwise, this set of zeros is given by 
\[
\left\{\alpha_k^{(n)} \equiv \nu_n-\pi-2\pi k,\ \forall k=0,\dots,k_0\right\},
\]
where $k_0 = \max\{k\in\na: \alpha_k^{(n)} > \nu_n - c_r\}=\max\{k\in\na: c_r> \pi+2k\pi\}$ and does not depend on $n$. Let us fix $\delta>0$ small to be chosen later, given $k\in\{0,\dots,k_0\}$ and $\alpha\in [\nu_n-c_r, \nu_n[ \cap ]\alpha_k^{(n)}-\delta, \alpha_k^{(n)}+\delta[$, using that $\nu_n-\alpha < c_r$ and that $\alpha_k^{(n)}< \alpha_0^{(n)}$, we have that
\begin{equation}
    \begin{split}
    \label{eq:Dem_lema_cota_sup_seno_u_3}
    h'_n(\alpha) &= \frac{-f_2(\alpha)(\nu_n-\alpha) + \int_\alpha^{\nu_n} f_2(t)\, \mathrm{d}t}{2\nu_n^r(\nu_n-\alpha)^2}
    \\
    & \geq \frac{-f_2\left(\alpha_0^{(n)} +\delta\right)c_r + \int_{\alpha_0^{(n)}+ \delta}^{\nu_n} f_2(t)\, \mathrm{d}t}{2 c_r^2 \nu_n^r} \\
    &= \frac{-\left(1-\cos\delta\right)\left(\alpha_0^{(n)} + \delta\right)^r c_r + \int_{\alpha_0^{(n)}+ \delta}^{\nu_n} f_2(t)\, \mathrm{d}t}{2 c_r^2 \nu_n^r} \\
    &= \frac{-C_\delta^{(1)}\left(\nu_n-\pi + \delta\right)^r + \int_{\nu_n-\pi + \delta}^{\nu_n} f_2(t)\, \mathrm{d}t}{2c_r^2 \nu_n^r},
    \end{split}
\end{equation}
where $C_\delta^{(1)}$ is a positive constant such that $C_\delta^{(1)}\to 0$ when $\delta\to 0$. With respect to the last integral in~\eqref{eq:Dem_lema_cota_sup_seno_u_3}, we can deduce that
\begin{equation}
    \begin{split}
    \label{eq:Dem_lema_cota_sup_seno_u_4}
    \int_{\nu_n-\pi + \delta}^{\nu_n} f_2(t)\, \mathrm{d}t &= \int_{\nu_n-\pi + \delta}^{\nu_n} t^r(1+\sin t)\, \mathrm{d}t 
    \\
    & \geq (\nu_n-\pi+\delta)^r \int_{\nu_n-\pi + \delta}^{\nu_n} (1+\sin t)\, \mathrm{d}t \\
    &= (\nu_n-\pi+\delta)^r (\pi-\delta + \sin \delta) 
    \\
    &= C_\delta^{(2)} (\nu_n-\pi+\delta)^r,
    \end{split}
\end{equation}
where $C_\delta^{(2)}$ is a positive constant such that $C_\delta^{(2)}\to \pi$ when $\delta\to 0$. Combining~\eqref{eq:Dem_lema_cota_sup_seno_u_3} and~\eqref{eq:Dem_lema_cota_sup_seno_u_4}, we obtain that
\[
h'_n(\alpha) \geq \frac{ \left(-C_\delta^{(1)} + C_\delta^{(2)}\right) (\nu_n-\pi+\delta)^r}{2 c_r^2 \nu_n^r}.
\] 
Therefore $h'(\alpha)\geq C$ for $\delta$ small enough and $n$ large, since $\frac{\nu_n-\pi-\delta}{\nu_n}\to 1$.

Thus, when $n$ is large we have proved that $s_0\notin \bigcup_{k=0}^{k_0}]\alpha_k^{(n)}-\delta, \alpha_k^{(n)}+\delta[$, which, in particular implies that $1+\sin s_0\geq C$ and we have \eqref{eq:Dem_lema_cota_sup_seno_u_5} proved. 

From~\eqref{eq:Dem_lema_cota_sup_seno_u_1},~\eqref{eq:Dem_lema_cota_sup_seno_u_2} and~\eqref{eq:Dem_lema_cota_sup_seno_u_5} we conclude that
\[
h_n(s) \geq C,\ \forall s\in [0,\nu_n[,\ \forall n\geq n_0,
\]
proving the claim~\eqref{eq:Dem_lema_cota_sup_seno_u_6} and ending the proof.
\end{proof}

Now we deal with an upper bound for the sequence $\{\lambda_n\}$.

\begin{proposition}
\label{prop:cota_superior_seno_u}
Assume that \eqref{hipg} is satisfied. Let $\nu_n$ be given by~\eqref{eq:def_M_n_seno_u} and let $\{\lambda_n\}$ be the sequence defined in~\eqref{eq:def_lambda_n} for the problem~\eqref{eq:Problema_seno_u}. Then there exist some $n_0\in\na$ and some $C>0$ independent from $n$ such that
\[
\lambda_n \leq C \nu_n^{1-r},\ \forall n\geq n_0.
\]
\end{proposition}

\begin{proof}
Let $R>0$ be as in the proof of Lemma~\ref{lema:radial} with $f_1(s)=\gamma\, s^r(1+\sin s)$, where $\gamma = \inf_{s\geq 0}g(s)>0$. From relation~\eqref{eq:Lema_radial_igualdad} of Lemma~\ref{lema:radial}, we define for each $\nu_n$ the number $\overline{\lambda}_n$ given by
\begin{equation}
\label{eq:Dem_prop_cota_sup_seno_u}
\hat f_1(\nu_n) = R\sqrt{\overline \lambda_n} \, \frac{2^{N-2}-1}{2^{N-\frac12}},
\end{equation}
where $\hat f_1$ is defined in~\eqref{eq:def_g}. Since $f_1(s)\leq s^r(1+\sin s)g(s)=f(s)$, by Lemma~\ref{lema:radial}, there exists a nonnegative weak subsolution $\underline{u}_n$ of $(Q_{\overline{\lambda}_n})$ such that $\|\underline{u}_n\|_{L^\infty(\Omega)} = \nu_n$. Since $\nu_n<\alpha_{n+1}$ (recall that $\alpha_n$ is defined in~\eqref{eq:def_alpha_n_seno_u}), we can apply the sub and supersolution method with $\underline{u}_n$ as subsolution and $\alpha_{n+1}$ as supersolution to guarantee the existence of a solution $u_n$ of $(Q_{\overline{\lambda}_n})$ such that $\underline{u}_n \leq u_n \leq \alpha_{n+1}$. It follows from Lemma~\ref{lema:no_existencia} that $\alpha_{n}< \|\underline{u}_n\|_{L^\infty(\Omega)} \leq \|u_n\|_{L^\infty(\Omega)} < \alpha_{n+1}$. 

From the definition~\eqref{eq:def_lambda_n} of $\lambda_n$, we get that $\lambda_n \leq \overline{\lambda}_n$ for every $n\in\na$. Finally, we can apply Lemma~\ref{lema:cota_superior_seno_u} to deduce from this inequality and from~\eqref{eq:Dem_prop_cota_sup_seno_u} that
\[
\lambda_n \leq \overline{\lambda}_n = C\hat f_1(\nu_n)^2 \leq c \nu_n^{1-r},\ \forall n\geq n_0,
\]
where $C>0$ is a constant and $n_0$ is given by Lemma~\ref{lema:cota_superior_seno_u}
\end{proof}

Once lower and upper bounds for $\{\lambda_n\}$ have been obtained, it is easy to estimate its asymptotic behavior.

\begin{proposition}
\label{prop:comp_asintotico_seno_u}
 Asumme that \eqref{hipg} holds true and consider the problem~\eqref{eq:Problema_seno_u}. The sequence $\{\lambda_n\}$ defined in~\eqref{eq:def_lambda_n} behaves as $n^{1-r}$ when $n$ tends to $\infty$, i.e., there exist some $c_1,c_2>0$ such that
\[
c_1 \leq \liminf_{n\to\infty}\frac{\lambda_n}{n^{1-r}} \leq \limsup_{n\to\infty}\frac{\lambda_n}{n^{1-r}} \leq c_2.
\]
As a consequence,
\begin{enumerate}[i)]
    \item if $0\leq r<1$, then $\lambda_n\to +\infty$,
    \item if $r=1$, then $c_1-\varepsilon\leq \lambda_n\leq c_2+\varepsilon$ for every $\varepsilon>0$ and $n$ large enough,
    \item  if $r>1$, then $\lambda_n\to 0$.
\end{enumerate}
\end{proposition}

\begin{proof}
The first inequality is an easy consequence of Proposition~\ref{prop:cota_inferior_seno_u} since
\[
\liminf_{n\to\infty} \frac{\lambda_n}{n^{1-r}} \geq C_1 \lim_{n\to\infty} \frac{\alpha_n}{n^{1-r}\alpha_{n+1}^r} = C_1 \lim_{n\to\infty} \frac{2\pi n - \frac{\pi}{2}}{n^{1-r} \left(2\pi n + \frac{3\pi}{2}\right)^r} = \frac{C_1}{(2\pi)^{r-1}}.
\]
The last inequality follows from Proposition~\ref{prop:cota_superior_seno_u} because
\[
\limsup_{n\to\infty}\frac{\lambda_n}{n^{1-r}} \leq C_2 \lim_{n\to \infty} \frac{\nu_n^{1-r}}{n^{1-r}} = C_2 \lim_{n\to \infty} \frac{\left(2\pi n + \frac{\pi}{2} \right)^{1-r}}{n^{1-r}} = \frac{C_2}{(2\pi)^{r-1}}.
\]
\end{proof}

Now we can easily prove Theorem~\ref{teor:seno_u}.

\begin{proof}[Proof of Theorem~\ref{teor:seno_u}]
Let $\{\lambda_n\}$ be the sequence defined in~\eqref{eq:def_lambda_n}. Regardless of the value of $r$, see that for every $m\in\na$ we can always take $\tilde{\lambda}_m = \max\{\lambda_1, \dots, \lambda_m\}$ and the very definition of the sequence $\{\lambda_n\}$ implies that $n(\lambda) \geq m$ for each $\lambda\geq \lambda_m$. This is essentially a consequence of Lemma~\ref{lema:existencia}.

On the other hand, the nonexistence of bifurcation points from infinity, as well as the existence of bifurcation points from infinity which are not branching points, is an immediate consequence of Proposition~\ref{prop:comp_asintotico_seno_u} and Proposition~\ref{prop:bifurcacion}.
\end{proof}

\subsection{The problem \eqref{eq:Problema_general} with $f$ having a vanishing sequence of zeros.}\hfill
\label{sec:seno_1/u}


 In this subsection we consider $f(t)=t^r\left(1+\sin \frac{1}{t}\right)g(t)$ for some $r > 0$ and $g$ satisfying \eqref{hipg} i.e., we study the problem
\begin{equation}
    \tag{$R_\lambda$}
    \begin{cases}
   -\Delta u = \lambda u^r \left(1+\sin \frac{1}{u}\right)g(u) &\text{ in } \Omega,\\
   u=0 &\text{ on } \partial\Omega.
    \end{cases}
\end{equation}

The main purpose in this subsection is to prove Theorem~\ref{teor:seno_1/u}. The strategy is the same as in the previous subsection and therefore we will focus in the main differences. 

%
Here, the sequence $\{\alpha_n\}$ of positive zeros of $f$, which is decreasing and convergent to zero, is given by:
\begin{equation}
    \label{eq:def_alpha_n_seno_1/u}
    \alpha_n = \frac{1}{-\frac{\pi}{2} + 2\pi n} ,\ \forall n\in\na.
\end{equation}
We observe that $f$ satisfies assumption~\ref{eq:hip_f_creciente_3} for $\sigma=\alpha_n$ and $n\in \mathbb{N}$, although for $r<2$ it does not verify~\ref{eq:hip_f_creciente}. 
The result related to the lower bound of $\{\lambda_n\}$ is stated below. Its proof is analogous to that of Proposition~\ref{prop:cota_inferior_seno_u}.

\begin{proposition}
\label{prop:cota_inferior_seno_1/u}
Assume that \eqref{hipg} is satisfied. Let $\alpha_n$ be given by~\eqref{eq:def_alpha_n_seno_1/u} and let $\{\lambda_n\}$ be the sequence defined in~\eqref{eq:def_lambda_n} for the problem~\eqref{eq:Problema_seno_1/u} . Then there exists some $C>0$ independent from $n$ such that
\[
C\frac{\alpha_{n+1}}{\alpha_{n}^r} \leq \lambda_n,\ \forall n\in\na.
\]
\end{proposition}

Although the strategy for proving the upper bound of $\{\lambda_n\}$ is the same, the calculations are somewhat more complicated. We first need a technical lemma related to $f_1(s)=s^r\left(1+\sin \frac{1}{s}\right)$.

\begin{lemma}
\label{lema:bachillerato}
If $r>-2$, then
\[
\left| \int_0^\nu t^r \sin\left( \frac{1}{t}\right) \, \mathrm{d}t \right| \leq (\pi+2) \nu^{r+2},\ \forall \nu>0.
\]
\end{lemma}
\begin{proof}
Fix $\nu>0$. We perform the change of variables $s=\frac{1}{t}$ and we get
\begin{equation}
\label{eq:Dem_lema_bach_1}
\int_0^\nu t^r \sin\left( \frac{1}{t}\right) \, \mathrm{d}t = \int_{\frac{1}{\nu}}^{+\infty} \frac{\sin s}{s^{r+2}}\, \mathrm{d}s.
\end{equation}
Let $N=\min\left\{n\in\na: \pi n \geq \frac{1}{\nu} \right\}$. We split the above integral as
\[
\int_{\frac{1}{\nu}}^{+\infty} \frac{\sin s}{s^{r+2}}\, \mathrm{d}s = \int_{\frac{1}{\nu}}^{\pi N} \frac{\sin s}{s^{r+2}}\, \mathrm{d}s + \sum_{k=N}^{+\infty} \int_{\pi k}^{\pi (k+1)} \frac{\sin s}{s^{r+2}}\, \mathrm{d}s.
\]
For each integral of the above summation, we perform the change of variables $\eta = \frac{s-k\pi}{\pi}$, obtaining that
\begin{equation}
\label{eq:Dem_lema_bach_2}
    \begin{split}
    \int_{\frac{1}{\nu}}^{+\infty} \frac{\sin s}{s^{r+2}}\, \mathrm{d}s &= \int_{\frac{1}{\nu}}^{\pi N} \frac{\sin s}{s^{r+2}}\, \mathrm{d}s + \sum_{k=N}^{+\infty}\frac{1}{\pi^{r+1}} \int_0^1 \frac{\sin (\pi\eta+ \pi k)}{(\eta+k)^{r+2}}\, \mathrm{d}\eta \\
    &= \int_{\frac{1}{\nu}}^{\pi N} \frac{\sin s}{s^{r+2}}\, \mathrm{d}s + \frac{1}{\pi^{r+1}} \sum_{k=N}^{+\infty} (-1)^k \int_0^1 \frac{\sin (\pi\eta)}{(\eta+k)^{r+2}}\, \mathrm{d}\eta.
    \end{split}
\end{equation}

On the one hand, we have
\begin{equation}
\label{eq:Dem_lema_bach_3}
\left| \int_{\frac{1}{\nu}}^{\pi N} \frac{\sin s}{s^{r+2}}\, \mathrm{d}s \right| \leq \nu^{r+2} \int_{\frac{1}{\nu}}^{\pi N} |\sin s| \leq \nu^{r+2} \int_{\pi(N-1)}^{\pi N} |\sin s|\leq \pi \nu^{r+2}.
\end{equation}

On the other hand, the sequence $\{a_k\}$ of positive numbers defined by
\[
a_k = \int_0^1 \frac{\sin (\pi\eta)}{(\eta+k)^{r+2}}\, \mathrm{d}\eta,\ \forall k\in\na
\]
converges to zero and is decreasing. Therefore we can apply the alternating series test to deduce that
\begin{equation}
\label{eq:Dem_lema_bach_4}
\begin{split}
\left| \sum_{k=N}^{+\infty} (-1)^k \int_0^1 \frac{\sin (\pi\eta)}{(\eta+k)^{r+2}}\, \mathrm{d}\eta \right| &\leq \int_0^1 \frac{\sin (\pi\eta)}{(\eta+N)^{r+2}}\, \mathrm{d}\eta \leq \frac{1}{N^{r+2}} \int_0^1 \sin(\pi \eta)\, \mathrm{d}\eta \\
&\leq (\pi \nu)^{r+2} \frac{2}{\pi} = 2\pi^{r+1} \nu^{r+2}.
\end{split}
\end{equation}

Combining~\eqref{eq:Dem_lema_bach_1},~\eqref{eq:Dem_lema_bach_2},~\eqref{eq:Dem_lema_bach_3} and~\eqref{eq:Dem_lema_bach_4} we get the proof finished.
\end{proof}

We now deal with the result which is similar to Lemma~\ref{lema:cota_superior_seno_u}. The proof in this case contains rather more tedious calculations. For this purpose, we define the points $\nu_n<\alpha_n$ in which $1+\sin \frac{1}{\nu_n}=2$, namely
\begin{equation}
    \label{eq:def_M_n_seno_1/u}
    \nu_n = \frac{1}{\frac{\pi}{2} + 2\pi n},\ \forall n\in\na.
\end{equation}
See that $\alpha_{n+1}<\nu_n<\alpha_{n}$ for every $n\in\na$, where $\{\alpha_n\}$ is defined in~\eqref{eq:def_alpha_n_seno_u}. We also observe now that $\nu_n\to 0$.

\begin{lemma}
\label{lema:cota_superior_seno_1/u}
Assume that \eqref{hipg} holds true and denote $\gamma = \inf_{s\geq 0}g(s)>0$. Let $f_1(t) = \gamma\,t^r\left(1+\sin \frac{1}{t}\right)$ with $r\geq 0$, $\nu_n$ defined by~\eqref{eq:def_M_n_seno_1/u} and let $\hat f_1$ be defined by~\eqref{eq:def_g}. Then there exists $n_0\in\na$ and $C>0$ independent of $n$ such that
\[
\hat f_1(\nu_n) \leq C \nu_n^\frac{1-r}{2},\ \forall n\geq n_0.
\]
\end{lemma}

\begin{proof} 
As shown in the proof of Lemma~\ref{lema:cota_superior_seno_u}, to prove this result it suffices to prove the following claim: there exists $n_0\in\na$ and $C>0$ such that
\begin{equation}
    \label{eq:Dem_lema_cota_sup_seno_1/u_6}
   h_n(s)\equiv \frac{\int_s^{\nu_n} f_2(t) \, \mathrm{d}t}{2\nu_n^r(\nu_n-s)} \geq C,\ \forall s\in [0,\nu_n[,\ \forall n\geq n_0,
\end{equation}
where $f_2(t) \equiv \gamma^{-1} f_1(t) = t^r\left(1+\sin \frac{1}{t}\right)$.

We easily deduce again that
\begin{equation}
    \label{eq:Dem_lema_cota_sup_seno_1/u_1}
    \lim_{s\to \nu_n} h_n(s) = \lim_{s\to \nu_n} \frac{-f_2(s)}{-2\nu_n^r} = \frac{f_2(\nu_n)}{2\nu_n^r} = 1,\ \forall n\in\na.
\end{equation}

Now we prove that 
\begin{equation}
    \label{eq:Dem_lema_cota_sup_seno_1/u_2}
    h_n(s)\geq C,\ \forall s\in[0, \nu_n-c_r \nu_n^2],
\end{equation}
for some $c_r>0$. 

Taking into account Lemma~\ref{lema:bachillerato} we deduce that
\begin{align*}
\left| \int_s^{\nu_n} f_2(t) \, \mathrm{d}t \right| & =\left| \frac{1}{r+1}(\nu_n^{r+1} - s^{r+1}) + \int_s^{\nu_n} t^r\sin \frac{1}{t}\, \mathrm{d}t \right|
\\
& \leq  \frac{1}{r+1}(\nu_n^{r+1} - s^{r+1}) + 2(\pi+2)\nu_n^{r+2}.
\end{align*}
%
Therefore we deduce that
\begin{align*}
h_n(s) &= \frac{1}{2(r+1)} \left[1+ \frac{s(\nu_n^r-s^r)}{\nu_n^r(\nu_n-s)}\right] + \frac{1}{2\nu_n^r(\nu_n-s)} \int_s^{\nu_n} t^r\sin \frac{1}{t}\, \mathrm{d}t\\
&\geq \frac{1}{2(r+1)} - \frac{2(\pi+2) \nu_n^{r+2}}{2\nu_n^r(\nu_n-s)}.
\end{align*}
In particular, we get that $h_n(s)\geq C>0$ whenever 
\[
s < \nu_n- \frac{2(\pi+2) \nu_n^2}{\frac{1}{r+1}-2C}
\]
being $C$ a positive constant small enough. Thus, we can choose $c_r=2(\pi+2)(r+1+\varepsilon)$ with $\varepsilon>0$ small. 

It remains only to show that 
\begin{equation}
    \label{eq:Dem_lema_cota_sup_seno_1/u_5}
    h_n(s)\geq C,\ \forall s\in[\nu_n-c_r \nu_n^2, \nu_n[.
\end{equation}
Observe again that if $h_n$ does not have any interior minimum point in that interval, then this is a direct consequence  of~\eqref{eq:Dem_lema_cota_sup_seno_1/u_1} and~\eqref{eq:Dem_lema_cota_sup_seno_1/u_2}. We assume now that $h_n$ has some minimum critical point $s_0\in ]\nu_n-c_r\nu_n^2, \nu_n[$ (observe that $s_0$ depends on $n$). In what follows we are going to prove that $h_n(s_0)\geq C$ for $n$ large, and this will imply that $h_n(s) \geq C$ for $s\in [\nu_n-c_r \nu_n^2, \nu_n[$ and $n$ large.

We first observe that $h_n'(s_0)=0$ i.e.
\[
\int_{s_0}^{\nu_n} f_2(t)\, \mathrm{d}t = f_2(s_0) (\nu_n-s_0).
\]
Thus $s_0$ satisfies that
\begin{equation*}
    h_n(s_0) = \frac{f_2(s_0)}{2\nu_n^r} = \frac{s_0^r \left(1+\sin \frac{1}{s_0} \right)}{2\nu_n^r} \geq \frac{1+\sin \frac{1}{s_0}}{2} \frac{\left(\nu_n-c_r \nu_n^2\right)^r}{\nu_n^r} > C \frac{1+\sin \frac{1}{s_0}}{2}
\end{equation*}
for $n$ large since $\nu_n\to 0$. 
Let us finally prove that $1+\sin \frac1{s_0}$ is bounded below by a positive constant. This is clear if the set of zeros of the function $1+\sin \frac1s$ does not intersect the interval $[\nu_n-c_r\nu_n^2, \nu_n[$. Otherwise, this set of zeros is given by 
\[
\left\{\alpha_k^{(n)} \equiv \frac{1}{\frac{1}{\nu_n}+\pi+2\pi k},\ \forall k=0,\dots,k_n\right\},
\]
where $k_n = \max\{k\in\na: \alpha_k^{(n)} > \nu_n-c_r \nu_n^2\}$. For $\delta>0$ small, we also define
\[
\beta_{k,\delta}^{(n)} = \frac{\alpha_k^{(n)}}{1+\delta  \alpha_k^{(n)}}, \quad \beta_{k,-\delta}^{(n)} = \frac{\alpha_k^{(n)}}{1-\delta  \alpha_k^{(n)}}.
\]
We observe that $\beta_{k,\delta}^{(n)} < \beta_{k,-\delta}^{(n)}$ and that we have chosen these points in such a way that they verify that $\sin\frac{1}{\beta_{k,\delta}^{(n)}} = \sin\left( \frac{1}{\alpha_k^{(n)}} + \delta \right)$ and that $\sin\frac{1}{\beta_{k,-\delta}^{(n)}} = \sin\left( \frac{1}{\alpha_k^{(n)}} - \delta \right)$. To make the next calculations easier to read, we denote
\[
\beta_k := \beta_{k,-\delta}^{(n)} = \frac{\nu_n}{((1+2k)\pi-\delta)\nu_n + 1}.
\]

Let $\alpha\in ]\beta_{k,\delta}^{(n)}, \beta_{k,-\delta}^{(n)}[ \cap [\nu_n-c_r \nu_n^2, \nu_n[$. Since $\nu_n-\alpha < c_r \nu_n^2$  and $\beta_k<\beta_0$, then we have that
\begin{equation}
    \begin{split}
    \label{eq:Dem_lema_cota_sup_seno_1/u_3}
    h'_n(\alpha) &= \frac{-f_2(\alpha)(\nu_n-\alpha) + \int_\alpha^{\nu_n} f_2(t)\, \mathrm{d}t}{2\nu_n^r(\nu_n-\alpha)^2} 
    \geq \frac{-f_2 \left(\beta_0 \right) c_r \nu_n^2 + \int_{\beta_0}^{\nu_n} f_2(t)\, \mathrm{d}t}{2 c_r^2 \nu_n^{r+4}} \\
    &= \frac{-\left(1-\cos\delta\right) c_r \beta_0^r \nu_n^2 + \int_{\beta_0}^{\nu_n} f_2(t)\, \mathrm{d}t}{2 c_r^2 \nu_n^{r+4}} 
    = \frac{-C_\delta^{(1)} \beta_0^r \nu_n^2 + \int_{\beta_0}^{\nu_n} f_2(t)\, \mathrm{d}t}{2 c_r^2 \nu_n^{r+4}},
    \end{split}
\end{equation}
where $C_\delta^{(1)}$ is a positive constant such that $C_\delta^{(1)}\to 0$ when $\delta\to 0$. With respect to the last integral in~\eqref{eq:Dem_lema_cota_sup_seno_1/u_3}, doing the change of variables $\eta = \frac{1}{t}$ we deduce that
\begin{equation}
    \begin{split}
    \label{eq:Dem_lema_cota_sup_seno_1/u_4}
    \int_{\beta_0}^{\nu_n} f_2(t)\, \mathrm{d}t &= \int_{\beta_0}^{\nu_n} t^r \left(1+\sin \frac{1}{t} \right)\, \mathrm{d}t \geq \beta_0^r \int_{\beta_0}^{\nu_n} \left(1+\sin \frac{1}{t} \right)\, \mathrm{d}t\\
    &= \beta_0^r \int_{1/\nu_n}^{1/\beta_0} \frac{1}{\eta^2} (1+\sin \eta)\, \mathrm{d}t \geq \beta_0^{r+2} \int_{1/\nu_n}^{1/\beta_0} (1+\sin \eta)\, \mathrm{d}t \\
    &
    = \beta_0^{r+2} \left(\frac{1}{\beta_0}-\frac{1}{\nu_n} - \cos\frac{1}{\beta_0} \right) = \beta_0^{r+2} \left(\frac{1}{\beta_0}-\frac{1}{\nu_n} - \sin\delta \right)\\
    &= \beta_0^{r+2} \left(\pi-\delta - \sin\delta \right) = C_\delta^{(2)} \beta_0^{r+2},
    \end{split}
\end{equation} 
where $C_\delta^{(2)}$ is a positive constant such that $C_\delta^{(2)}\to \pi$ when $\delta\to 0$. Combining~\eqref{eq:Dem_lema_cota_sup_seno_1/u_3} and~\eqref{eq:Dem_lema_cota_sup_seno_1/u_4}, we obtain that
\begin{equation*}
 h_n'(\alpha) \geq \frac{-C_\delta^{(1)} \beta_0^r \nu_n^2 + C_\delta^{(2)} \beta_0^{r+2}}{2 c_r^2 \nu_n^{r+4}} = \frac{-C_\delta^{(1)} [(\pi-\delta)\nu_n + 1]^2 + C_\delta^{(2)}}{2 c_r^2 [(\pi-\delta)\nu_n + 1]^{r+2}} \cdot \frac{1}{\nu_n^2},
\end{equation*}
so $h'(\alpha)\geq C$ as long as $\delta$ is small and $n$ large since $\nu_n\to 0$.

Thus, when $n$ is large we have proved that $s_0\notin \bigcup_{k=0}^{k_n} ]\beta_{k,\delta}^{(n)}, \beta_{k,-\delta}^{(n)}[$, which, in particular implies that $1+\sin \frac1{s_0}\geq C$ and we have \eqref{eq:Dem_lema_cota_sup_seno_1/u_5} proved. 

%
%

From~\eqref{eq:Dem_lema_cota_sup_seno_1/u_1},~\eqref{eq:Dem_lema_cota_sup_seno_1/u_2} and~\eqref{eq:Dem_lema_cota_sup_seno_1/u_5} we conclude that
\[
h_n(s) \geq C,\ \forall s\in [0,\nu_n[
\]
for $n$ large which proves the claim~\eqref{eq:Dem_lema_cota_sup_seno_1/u_6} and finishes the proof.
\end{proof}

Arguing as in Proposition~\ref{prop:cota_superior_seno_u} we deduce from Lemma~\ref{lema:cota_superior_seno_1/u} the following result.

\begin{proposition}
\label{prop:cota_superior_seno_1/u}
Assume that \eqref{hipg} is satisfied. Let $\nu_n$ be given by~\eqref{eq:def_M_n_seno_1/u} and let $\{\lambda_n\}$ be the sequence defined in~\eqref{eq:def_lambda_n} for the problem~\eqref{eq:Problema_seno_1/u}. Then there exist some $n_0\in\na$ and some $C>0$ independent from $n$ such that
\[
\lambda_n \leq C \nu_n^{1-r},\ \forall n\geq n_0.
\]
\end{proposition}

Now we estimate the asymptotic behavior of $\{\lambda_n\}$ for problem~\eqref{eq:Problema_seno_1/u}. Its behavior compared to the one of the sequence $\{\lambda_n\}$ for problem~\eqref{eq:Problema_seno_u} is just the inverse (compare the following result with Proposition~\ref{prop:comp_asintotico_seno_u}).

\begin{proposition}
\label{prop:comp_asintotico_seno_1/u}
Assume that \eqref{hipg} holds true and consider the problem~\eqref{eq:Problema_seno_1/u}. The sequence $\{\lambda_n\}$ defined in~\eqref{eq:def_lambda_n} behaves as $n^{r-1}$ when $n$ tends to $\infty$, i.e., there exist some $c_1,c_2>0$ such that
\[
c_1 \leq \liminf_{n\to\infty}\frac{\lambda_n}{n^{r-1}} \leq \limsup_{n\to\infty}\frac{\lambda_n}{n^{r-1}} \leq c_2.
\]
As a consequence,
\begin{enumerate}[i)]
    \item if $0< r<1$, then $\lambda_n\to 0$,
    \item if $r=1$, then $c_1-\varepsilon\leq \lambda_n\leq c_2+\varepsilon$ for every $\varepsilon>0$ and $n$ large enough,
    \item  if $r>1$, then $\lambda_n\to +\infty$.
\end{enumerate}
\end{proposition}

Now we can easily give the proof of Theorem~\ref{teor:seno_1/u} that it is essentially the same as the one of Theorem~\ref{teor:seno_u}.

\begin{proof}[Proof of Theorem~\ref{teor:seno_1/u}]
The nonexistence of bifurcation points from zero, as well as the existence of bifurcation points  from zero which are not branching points, are immediate consequences of Proposition~\ref{prop:comp_asintotico_seno_1/u} and Proposition~\ref{prop:bifurcacion}.
\end{proof}

\begin{remark}
\begin{enumerate}
 \item Both Theorem~\ref{teor:seno_u} and Theorem~\ref{teor:seno_1/u} are also true if we add a scale parameter inside the sine function. Due to the homogeneity of function $t^r$, this can be deduced from our results using standard scaling arguments.

\item  Both Theorem~\ref{teor:seno_u} and Theorem~\ref{teor:seno_1/u} show that, in problems~\eqref{eq:Problema_seno_u} and~\eqref{eq:Problema_seno_1/u}, when $\lambda$ belongs to the interior of the set of bifurcation points (either from infinity or from zero), the bifurcation is at the same time vertical, subcritical and supercritical. 
\end{enumerate}

\end{remark}

\section{Behavior of the solutions below the first positive zero or above the last one} 
\label{firstsecond}

Our aim is to complete the solution set diagram of~\eqref{eq:Problema_general}. Therefore, we study the existence of continua of solutions whose $\lio$-norms are below the first positive zero or above the last one. The shape of these continua will depend strongly on the behavior of the nonlinearity either at zero or at infinity.  We denote by $Z_f = \{z > 0: f(z)=0\}$, $\tilde \alpha \equiv \inf Z_f$ and $\tilde \beta \equiv \sup Z_f$, and we assume in this section that $\tilde\alpha$ and $\tilde\beta$ are positive and finite. 

We now detail all the possible cases, which of them are solved in the literature, which are treated by us in the present work, and, to the best of our knowledge, an open problem that is still pending. 

With respect to the behavior of solutions less than $\tilde\alpha$ we distinguish the case where $f(0)>0$ and the case where $f(0)=0$:

\begin{itemize}
    \item The case in which $f(0)>0$. It is studied in \cite{Amb-Hess}.  The existence of an unbounded continuum of positive solutions less than $\tilde\alpha$ emanating from $(0,0)$ is an easy consequence of~\cite[Theorem~3.2]{Rab}. Observe that there cannot be any bifurcation point from zero, since the zero function is not a solution of~\eqref{eq:Problema_general} when $\lambda>0$.
    
    \item In the case where $f(0)=0$ we have the following cases: 
    
     If the function $f$ is superlinear at zero, that is, $\lim_{s\to 0^+} \frac{f(s)}{s} = +\infty$, in~\cite{A-C-Pellacci}, the authors studied a problem more general than~\eqref{eq:Problema_general}, dealing with a quasilinear operator and nonlinearity that also depends on $x$. For the case at hand, they showed that problem~\eqref{eq:Problema_general} has an unbounded continuum of positive solutions less than $\tilde\alpha$ emanating from $(0,0)$ and that 0 is the only possible bifurcation point from zero. The strategy is to prove that at the bifurcation points a change of index occurs in order to apply the global bifurcation theorem of Rabinowitz~\cite[Theorem~1.3]{Rab}. 
    
     In the case in which $f$ is asymptotically linear at zero, that is, $\lim_{s\to 0^+} \frac{f(s)}{s} := m_0>0$, it is studied in \cite{Amb-Hess}. They showed that there is an unbounded continuum of positive solutions less than $\tilde\alpha$ emanating from $(\lambda_0, 0)$, where $\lambda_0=\frac{\lambda_1}{m_0}$, and this $\lambda_0$ is the only bifurcation point from zero. Here, the projection of the continuum on the $\lambda$-axis has the form $(\lambda_0, +\infty)$ or $[\lambda_{\tilde\alpha},+\infty)$ for some $\lambda_{\tilde\alpha}>0$.
     
     We study the case $f$ sublinear at zero in Subsection~\ref{superlinear} below.
\end{itemize}

With respect to the behavior of solutions with $L^\infty(\Omega)$-norm greater than $\tilde\beta$, this strongly depends on $f$: 
\begin{itemize}
    \item We study the case in which the function $f$ is sublinear at infinity in Subsection~\ref{sublinear} below.
    \item The case $f$ asymptotically linear at infinity, i.e. $\lim_{s\to+\infty}\frac{f(s)}{s} := m_\infty>0$ is studied in \cite{Amb-Hess}  where the authors proved the existence of an unbounded continuum of positive solutions meeting $(\lambda_\infty,\infty)$ with $\lio$-norm greater than $\tilde\beta$, where $\lambda_\infty:=\frac{\lambda_1}{m_\infty}$. Moreover, they showed that $\lambda_\infty$ is the only bifurcation point from infinity and that the projection of the continuum on the $\lambda$-axis has the form $(\lambda_\infty,+\infty)$ or $[\lambda_{\tilde\beta},+\infty)$ for some $\lambda_{\tilde\beta}>0$. This result is again a consequence of the global bifurcation theorem of Rabinowitz~\cite[Theorem~1.3]{Rab}. 
    \item When the function $f$ is superlinear at infinity, this is still an open problem that we describe in Subsection~\ref{openproblem} below.
\end{itemize}

In the next subsections, we deal with the cases where, up to our knowledge, there are no global results related with the existence of a continuum of positive solutions which give a global picture. 

Hereunder, we prove some results in this direction.

\subsection{The case in which $f(0)=0$ and $f$ is sublinear at zero. No bifurcation from zero. }
\label{superlinear}
\mbox{ }

In this case, we prove the following result.

\begin{theorem}
Let $f$ be a nonnegative function such that $f(0)=0$, $\tilde \alpha\in (0,+\infty)$ and $\lim_{s\to 0^+} \frac{f(s)}{s} = 0$. Assume that~\ref{eq:hip_f_creciente} holds true for $\sigma=\tilde\alpha$. Then there is no bifurcation point from zero and there exists an unbounded continuum $\Sigma \subset \mathcal{S}$ with $\subset$-shape such that
\begin{enumerate}[i)]
\item $\|u\|_{L^\infty(\Omega)}\in (0,\tilde\alpha)$ for every $(\lambda, u)\in \Sigma$.
\item $0< \lambda_{\tilde\alpha} = \min\left( \mathrm{Proj}_{\lambda}\Sigma \right)$, where $\lambda_{\tilde\alpha}$ is defined as \[
\lambda_{\tilde\alpha} := \inf \{\lambda\geq 0: \textrm{\eqref{eq:Problema_general} admits solution $u$ with } 0 < \|u\|_{L^\infty(\Omega)}<\tilde \alpha \}
\]
\item For every $\lambda> \lambda_{\tilde\alpha}$ there exist at least two solutions $u_1, u_2$ of \eqref{eq:Problema_general}  such that $u_1\not\equiv u_2$ and 
$(\lambda,u_1),(\lambda, u_2)\in \Sigma$. 
\end{enumerate}
\end{theorem}

\begin{proof}
First, we prove the nonexistence of bifurcation points from zero. Since $\lim_{s\to 0^+} \frac{f(s)}{s} = 0$, then for every $\varepsilon>0$ there exists some $\delta>0$ such that $f(s)<\varepsilon s$ for $s\in[0,\delta[$. Let $u$ be a solution of~\eqref{eq:Problema_general} with $\|u\|_\lio < \delta$. Taking the first eigenfunction $\varphi_1>0$ as test function in~\eqref{eq:Problema_general} we obtain that
\[
\lambda_1 \into \varphi_1 u = \into \nabla\varphi_1 \nabla u = \lambda \into f(u)\varphi_1 \leq \varepsilon\lambda \into u\varphi_1,
\]
where $\lambda_1$ denotes the first eigenvalue, and thus $\frac{\lambda_1}{\varepsilon} \leq \lambda$. Therefore, there is no solution $u$ of~\eqref{eq:Problema_general} with $\|u\|_\lio<\delta$ for $\lambda<\frac{\lambda_1}{\varepsilon}$. Since $\varepsilon$ can be arbitrarily small, we get the desired result (observe that $\delta$ decreases when $\varepsilon$ does).

Now, taking into account what we have proved, we can follow the steps of Lemma~\ref{lema:existencia} replacing $\tilde f$ by 
\begin{equation*}
\tilde f_0(s)=\begin{cases}
f(s) & \mbox{ if }  s\in[0, \tilde\alpha], \\
0 & \mbox{ if } s\notin[0, \tilde\alpha]\\
\end{cases} 
\end{equation*}
to get that the set
 \[
\{\lambda\geq 0: \textrm{\eqref{eq:Problema_general} admits solution $u$ with } 0 < \|u\|_{L^\infty(\Omega)}<\tilde \alpha \}
\]
is a nonempty closed set of the form $[\lambda_{\tilde \alpha},+\infty)$ for some $\lambda_{\tilde \alpha}>0$. Thus, we can consider the maximal solution $v$ of $(P_{\lambda_{\tilde\alpha}})$ and we can argue as in the proof of Theorem~\ref{teor:continuo} with $\beta$ replaced by $\tilde\alpha$.
\end{proof}

\begin{figure}[ht]
\centering
\begin{subfigure}[t]{.33\textwidth}
  \centering
  \includegraphics[scale=0.42]{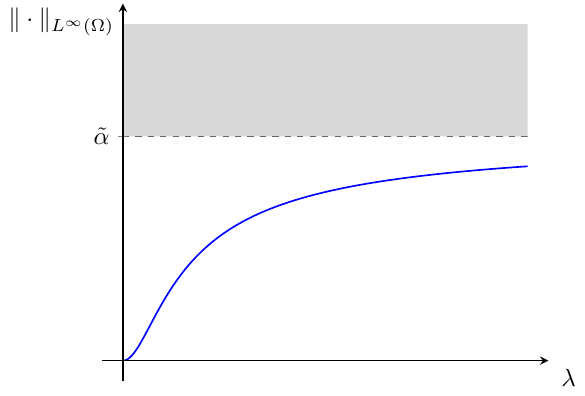}
  \caption{$f(0)>0$ or $f(0)=0$ and $f$ superlinear at 0}
\end{subfigure}%
\begin{subfigure}[t]{.33\textwidth}
  \centering
  \includegraphics[scale=0.42]{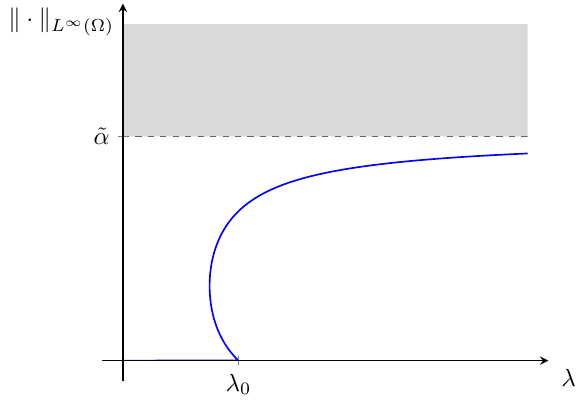}
  \caption{$f$ linear at 0}
\end{subfigure}%
\begin{subfigure}[t]{.33\textwidth}
  \centering
  \includegraphics[scale=0.42]{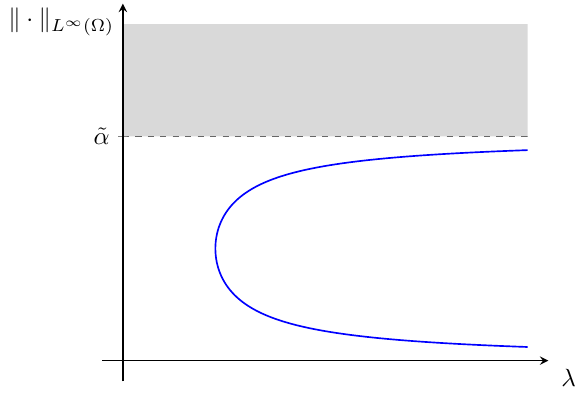}
  \caption{$f$ sublinear at 0}
\end{subfigure}
\caption{Sketch of the solution set diagrams of problem~\eqref{eq:Problema_general} below the first positive zero depending on behavior of $f$ at 0.}
\end{figure}

\subsection{The case in which $f$ is sublinear at infinity. No bifurcation from infinity.}
\label{sublinear}
\mbox{ }

Our main result in this case is the following.
\begin{theorem}
Let $f$ be a nonnegative function such that $\lim_{s\to +\infty} \frac{f(s)}{s} = 0$. Assume that~\ref{eq:hip_f_creciente} holds true in $(0,+\infty)$. Then there is no bifurcation point from infinity and there exists an unbounded continuum $\Sigma \subset \mathcal{S}$ with $\subset$-shape such that
\begin{enumerate}[i)]
\item $\|u\|_{L^\infty(\Omega)}\in (\beta, +\infty)$ for every $(\lambda, u)\in \Sigma$.
\item $0< \lambda_{\tilde\beta} = \min\left( \mathrm{Proj}_{\lambda}\Sigma \right)$, where $\lambda_{\tilde\beta}$ is defined as \[
\lambda_{\tilde\beta} := \inf \{\lambda\geq 0: \textrm{\eqref{eq:Problema_general} admits solution $u$ with } \|u\|_{L^\infty(\Omega)}>\tilde \beta \}
\]
\item For every $\lambda> \lambda_{\tilde\beta}$ there exist at least two solutions $u_1, u_2$ of \eqref{eq:Problema_general}  such that $u_1\not\equiv u_2$ and 
$(\lambda,u_1),(\lambda, u_2)\in \Sigma$. 
\end{enumerate}
\end{theorem}

\begin{proof}
First, we prove the nonexistence of bifurcation points from infinity. Arguing by contradiction, assume that there exists a sequence $u_n$ of solutions of $(P_{\lambda_n})$ such that $\|u_n\|_\lio \to +\infty$, where $\lambda_n\to \tilde\lambda\geq 0$. Let $v_n:=\|u_n\|_\lio^{-1} u_n$. Observe that $\|v_n\|_\lio=1$ and that $v_n$ satisfies the problem
\begin{equation*}
    \begin{cases}
   -\Delta v_n = \lambda_n \frac{f(u_n)}{\|u_n\|_\lio} &\text{ in } \Omega,\\
   v_n=0 &\text{ on } \partial\Omega.
    \end{cases}
\end{equation*}
As $f$ is sublinear at infinity, the right-hand side of the above equation is bounded in $\lio$ and converges uniformly to 0. It is straightforward to show that a subsequence of $v_n$ converges weakly in $\huz$ and uniformly to some $\tilde v$. Now observe that on the one hand, we have $\|\tilde v\|_\lio =1$ and, on the other hand, we have that $\tilde v$ is solution of
\begin{equation*}
    \begin{cases}
   -\Delta \tilde v = 0 &\text{ in } \Omega,\\
   \tilde v=0 &\text{ on } \partial\Omega,
    \end{cases}
\end{equation*}
and thus $\tilde v =0$, a contradiction.

Now, we show that for every $\lambda>0$ problem~\eqref{eq:Problema_general} has a strict supersolution $\overline{v}_\lambda$ which is greater than any possible subsolution $\underline u_\lambda$ of~\eqref{eq:Problema_general}, regardless of whether there exists this subsolution $\underline u_\lambda$ or not. Although this is a well-known fact in the literature (for example, see \cite{Fig}), we include here the proof for the convenience of the reader.

Fix $\lambda>0$ and let $m_1>0$ be such that $m_1\lambda < \lambda_1$. Since $\lim_{s\to +\infty} \frac{f(s)}{s} = 0$ and $f$ is continuous, then $f(s)\leq m_1s + m_2$ when $s\geq 0$ for some $m_2>0$. Using that $\lambda m_1< \lambda_1$, problem
\begin{equation}
\label{eq:Dem_sublineal_inf_1}
    \begin{cases}
   -\Delta v = \lambda (m_1 v + m_2) &\text{ in } \Omega,\\
   v=0 &\text{ on } \partial\Omega,
    \end{cases}
\end{equation}
has a solution $\overline{v}_\lambda$ and, in fact, $\overline{v}_\lambda>0$ as a consequence of the strong maximum principle~\cite[Theorem~1.14]{Fig}. Observe also that if $\underline u_\lambda$ is a positive subsolution of~\eqref{eq:Problema_general}, then it is also a subsolution of~\eqref{eq:Dem_sublineal_inf_1}. In this way, $\overline{v}_\lambda-\underline u_\lambda$ satisfies problem
\begin{equation*}
    \begin{cases}
   -\Delta v \geq \lambda m_1 v &\text{ in } \Omega,\\
   v=0 &\text{ on } \partial\Omega,
    \end{cases}
\end{equation*}
and again the maximum principle leads to $\underline u_\lambda\leq \overline{v}_\lambda$, as we wanted to prove.

Consider now problem~\eqref{eq:Problema_general} with $f$ replaced by
\begin{equation*}
\tilde f_\infty(s)=\begin{cases}
0 & \mbox{ if }  s < \tilde\beta, \\
f(s) & \mbox{ if } s \geq \tilde\beta\\
\end{cases} 
\end{equation*}
and consider also the energy functional associated
\begin{equation*}
    \tilde{I}_\lambda(u) = \frac{1}{2} \into |\nabla u|^2 - \lambda\into \tilde{F}_\infty(u),\ \forall u\in\huz,
\end{equation*}
where $\tilde{F}_\infty(s) = \int_0^s \tilde{f}_\infty(t)\, \mathrm{d}t$. Since $f$ is sublinear at infinity, this functional is coercive and weakly lower semicontinuous. Thus, $\tilde{I}_\lambda$ achieves its infimum. Observe that we can take $u\in\huz$ so that $\tilde{F}_\infty(u)>0$ and, in this way, $\tilde{I}_\lambda(u)<0$ for large $\lambda$. Therefore, problem~\eqref{eq:Problema_general} with $f$ replaced by $\tilde f_\infty$ has a non trivial solution $\underline v_\lambda$ which has to be positive with $\lio$-norm greater than $\tilde\beta$ by Lemma~\ref{lem:AmbHess}.

Clearly $\underline v_\lambda$ is a subsolution of~\eqref{eq:Problema_general} for large fixed $\lambda$ and consequently $\underline v_\lambda\leq \overline v_\lambda$. Thus, we can apply the sub-supersolution method with $\underline v_\lambda$ and $\overline v_\lambda$ to ensure the existence of a positive solution $v_\lambda$ of~\eqref{eq:Problema_general} with $\lio$-norm greater than $\tilde\beta$.

Since there are no bifurcation points from infinity, the solutions of~\eqref{eq:Problema_general} are a priori bounded for $\lambda$ in compact intervals of $\re_0^+$. Due to this, we can argue as in Lemma~\ref{lema:existencia} to prove that
\[
\{\lambda\geq 0: \textrm{\eqref{eq:Problema_general} admits solution $u$ with } \|u\|_{L^\infty(\Omega)}>\tilde \beta \}
\]
is a nonempty closed set with the form $[\lambda_{\tilde \beta},+\infty)$ for some $\lambda_{\tilde \beta}>0$. Moreover, since for every $\lambda\in[\lambda_{\tilde \beta},+\infty)$ the supersolution $\overline{v}_\lambda$ is greater than every solution of~\eqref{eq:Problema_general}, we can apply an iterative scheme (see \cite{amann}) to obtain the existence of a maximal solution of~\eqref{eq:Problema_general} which necessarily has $\lio$-norm greater than $\tilde{\beta}$.

Now, we can argue as in Theorem~\ref{teor:continuo} with $\underline{u}$ as the maximal solution of $(P_{\lambda_{\tilde\beta}})$ and $\overline{u}$ as $\overline{v}_{\overline{\lambda}}$ for the $\overline{\lambda}> \lambda_{\tilde\beta}$ fixed in the proof.
\end{proof}

\subsection{Open problem. Bifurcation from infinity. The case in which $f$ is superlinear and subcritical.}
\label{openproblem}
\mbox{ }

When $f$ is superlinear and subcritical at infinity (i.e. $\lim_{s\to +\infty} \frac{f(s)}{s} = +\infty$ and $\lim_{s\to +\infty} f(s)s^{-\frac{N+2}{N-2}} = 0$), the question we address is much more complicated. Lions proved in~\cite[Theorem~3.1]{Lions} using a topological degree argument that if the domain $\Omega$ is convex and the technical condition
\begin{equation}
\label{eq:superlineal_condicion_tecnica}
    \limsup_{s\to+\infty} \frac{sf(s)-\theta F(s)}{s^2f(s)^{2/N}}\leq 0 \text{ for some } 0<\theta<\frac{2N}{N-2}
\end{equation}
is satisfied, then problem~\eqref{eq:Problema_general} has a solution satisfying $\|u\|_\lio>\beta$. This technical assumption~\eqref{eq:superlineal_condicion_tecnica}, which was first introduced in~\cite{Fig-Lions-N}, was only used to obtain a priori bounds. Lions conjectured that neither the convexity of the domain nor condition~\eqref{eq:superlineal_condicion_tecnica} were necessary to obtain this existence result. This is the case when $f$ behaves at infinity as $t^r$ for $1<r<\frac{N+2}{N-2}$, as proved in~\cite{GS}.


%

To our knowledge, there is no result dealing with the existence of a global continuum of positive solutions with $\lio$-norm greater than $\beta$. We think that it exists a continuum of solutions meeting $(0,\infty)$ and whose projection on the $\lambda$-axis is unbounded, but we left this issue as an open question.

\begin{figure}[ht]
\centering
\begin{subfigure}{0.33\textwidth}
  \centering
  \includegraphics[scale=0.43]{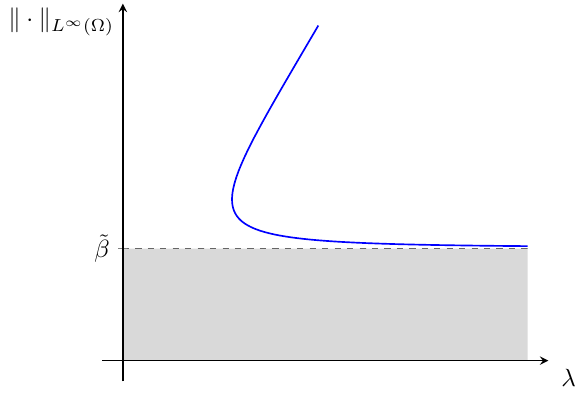}
  \caption{$f$ sublinear at infinity}
\end{subfigure}%
\begin{subfigure}{0.33\textwidth}
  \centering
  \includegraphics[scale=0.43]{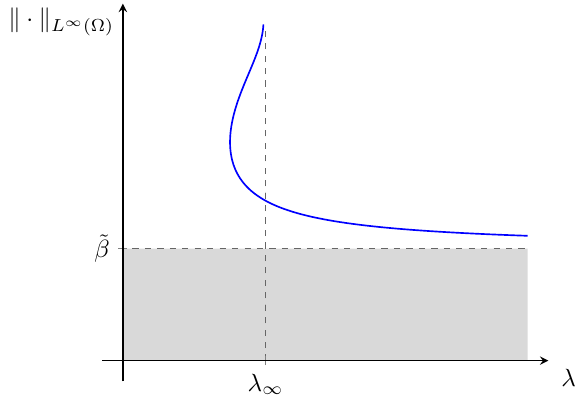}
  \caption{$f$ linear at infinity}
\end{subfigure}%
\begin{subfigure}{0.33\textwidth}
  \centering
  \includegraphics[scale=0.43]{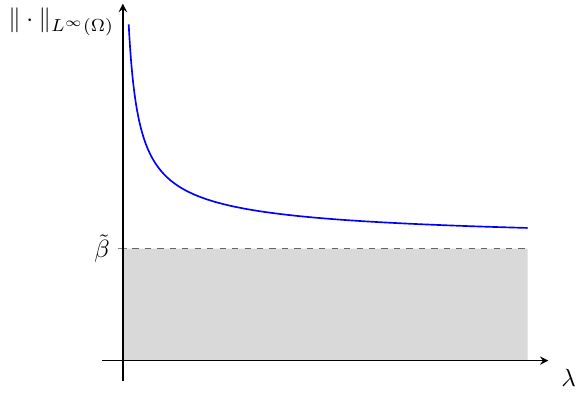}
  \caption{$f$ superlinear at infinity}
\end{subfigure}
\caption{Sketch of the bifurcation diagrams of problem~\eqref{eq:Problema_general} above the last positive zero depending on the behavior of $f$ at infinity.}
\end{figure}

\centerline{\bf{ Author's statements}}

 \noindent{\bf Funding information.}
 \noindent Research supported by (MCIU) Ministerio de Ciencia, Innovaci\'on y Universidades, Agencia Estatal de Investigaci\'on (AEI) and Fondo Europeo de Desarrollo Regional under Research Project PID2021-122122NB-I00 (FEDER). First  and second author supported by Junta de Andaluc\'ia FQM-194.  First author supported by CDTIME.  Second author has been funded by the FPU predoctoral fellowship of the Spanish Ministry of Universities (FPU21/04849). Third author supported by Junta de Andaluc\'{\i}a FQM-116.
 
 \noindent{\bf Authors contributions.}
 \noindent All authors have contributed more or less equally to the preparation of this manuscript. All authors have accepted responsibility for the entire content of this manuscript and consented to its
submission to the journal, reviewed all the results and approved the final version of the manuscript.

 \noindent{\bf Conflict of interest.}
 \noindent Authors state no conflict of interest.

\end{document}